\documentclass[12p]{amsart}
\usepackage{amssymb}
\usepackage{amsmath}
\usepackage{amsfonts}
\usepackage{geometry}
\usepackage{graphicx}
\usepackage{mathrsfs,amssymb}
\usepackage{cancel}

\usepackage{hyperref}
\usepackage{cleveref}

\theoremstyle{plain}

\newtheorem{claim}{Claim}

\newtheorem{corollary}{Corollary}

\newtheorem{lemma}{Lemma}

\newtheorem{proposition}{Proposition}
\newtheorem{remark}{Remark}

\newtheorem{theorem}{Theorem}

\numberwithin{equation}{section}

\begin{document}
\title{Global well--posedness for the mass--critical nonlinear Schr{\"o}dinger equation on $\mathbb{T}$}

\author{Benjamin Dodson}
\date{\today}

\begin{abstract}
We prove a global well--posedness result for the quintic NLS on $\mathbb{T}$ for initial data in $H^{s}(\mathbb{T})$, $s > 1/3$. This improves the previous best bound of $s > 2/5$.
\end{abstract}

\maketitle

\section{Introduction}
In this paper we consider the Cauchy problem for the periodic nonlinear Schr{\"o}dinger equation:
\begin{equation}\label{1.1}
i \partial_{t} u + \partial_{xx} u - |u|^{4} u = 0, \qquad u : \mathbb{T} \times [0, T] \rightarrow \mathbb{C}, \qquad u(x, 0) = \phi(x) \in H^{s}(\mathbb{T}),
\end{equation}
where $\mathbb{T} = \mathbb{R} / \mathbb{Z}$ is the torus.\medskip

The quintic nonlinear Schr{\"o}dinger equation on $\mathbb{R}$ is called mass--critical, since the scaling
\begin{equation}\label{1.2}
u(t, x) \mapsto \lambda^{-1/2} u(x/\lambda, t/\lambda^{2}), \qquad \lambda > 0,
\end{equation}
leaves the solution to
\begin{equation}\label{1.3}
i \partial_{t} u + \partial_{xx} u - |u|^{4} u = 0, \qquad u : \mathbb{R} \times [0, T] \rightarrow \mathbb{C},
\end{equation}
invariant, as well as the mass $\int_{\mathbb{R}} |u(x, t)|^{2} dx$. Since the mass,
\begin{equation}\label{1.4}
M(u(t)) = \int |u(t, x)|^{2} dx = M(u(0)),
\end{equation}
is a conserved quantity, and $(\ref{1.3})$ also has a conserved energy
\begin{equation}\label{1.5}
E(u(t)) = \int [\frac{1}{2} |u_{x}(x, t)|^{2} + \frac{1}{6} |u(x,t)|^{6}] dx = E(u(0)),
\end{equation} 
that is coercive, $\| u(t) \|_{\dot{H}^{1}}^{2} \lesssim E(u(t))$, the initial value problem $(\ref{1.3})$ is globally well--posed and scattering for any $u(0) \in L^{2}$ (see \cite{dodson2019defocusing} for a complete description of the resolution of this issue).\medskip

The Cauchy problem on the torus, or periodic problem, $(\ref{1.1})$, possesses many common elements with the problem on $\mathbb{R}$, and thus the results for $(\ref{1.3})$ guide the intuition of what should be attainable for $(\ref{1.1})$. For example, equation $(\ref{1.1})$ also conserves the quantities mass,
\begin{equation}\label{1.6}
M(u(t)) = \int_{\mathbb{T}} |u(x, t)|^{2} dx = M(u(0)),
\end{equation}
and energy,
\begin{equation}\label{1.7}
E(u(t)) = \int_{\mathbb{T}} \frac{1}{2} |u_{x}(x, t)|^{2} + \frac{1}{6} |u(x, t)|^{6} dx = E(u(0)).
\end{equation}
It is also known that $(\ref{1.1})$ is locally well--posed in $H^{s}(\mathbb{T})$ for $s > 0$, see Proposition $5.73$ of \cite{bourgain1993fourier}. An immediate consequence of this result combined with $(\ref{1.7})$ is that $(\ref{1.1})$ is globally well--posed for $\phi \in H^{1}(\mathbb{T})$.\medskip

However, in examining the local well--posedness results on $\mathbb{T}$ and comparing them to local well--posedness results on $\mathbb{R}$, some important differences between the behavior of $(\ref{1.1})$ and $(\ref{1.3})$ become apparent. Observe that local (and global) well--posedness holds for $(\ref{1.3})$ when $\phi \in H^{s}(\mathbb{R})$ and $s \geq 0$, but local well--posedness holds for $(\ref{1.1})$ when $\phi \in H^{s}(\mathbb{T})$ and $s > 0$. This is probably not a mere technical issue that should be resolvable by improved analysis of the equation. Rather, the remarkable result of \cite{herr2026global} proved (among other things) that the mass--critical problem on $\mathbb{T}^{2}$,
\begin{equation}\label{1.8}
i u_{t} + \Delta u = |u|^{2} u, \qquad u(x, 0) \in H^{s}(\mathbb{T}^{2}),
\end{equation}
is ill--posed in $H^{0}(\mathbb{T}^{2}) = L^{2}(\mathbb{T}^{2})$. This is very likely to be true in $(\ref{1.1})$ as well.\medskip

Furthermore, \cite{herr2026global} also proved the remarkable result that $(\ref{1.8})$ is globally well--posed in $H^{s}(\mathbb{T}^{2})$ when $s > 0$. The proof uses the Strichartz estimates of \cite{herr2024strichartz} combined with a concentration compactness argument. Proving an analogous result for $(\ref{1.1})$, or for the mass--critical problem in higher dimensions $d > 3$ would be a very nice result. In this paper we make partial progress in this direction using the I--method.
\begin{theorem}\label{t1.1}
The Cauchy problem $(\ref{1.1})$ is globally well--posed in $H^{s}(\mathbb{T})$ for $s > \frac{1}{3}$.
\end{theorem}
Global well--posedness for $(\ref{1.1})$ with $\phi \in H^{s}(\mathbb{T})$, $s \geq 1$, was extended to $s > \frac{1}{2}$ by \cite{bourgain2004remark} using the normal form reduction combined with the I-method. The normal form reduction removes the strongly non--resonant part of the Hamiltonian and reduces the nonlinearity to its essential part. This result was subsequently extended to $s > \frac{4}{9}$ by \cite{de2007global} and then to $s > \frac{2}{5}$ by \cite{li2011global}. See \cite{schippa2024improved} and \cite{mcconnell2026lattice} for small data global well--posedness results for $(\ref{1.1})$. Recently, \cite{skouloudis2026strichartz} extended these small data results to prove global well--posedness for solutions to $(\ref{1.1})$ with $u_{0} \in H^{s}(\mathbb{T})$, $s > 0$, and $\| u_{0} \|_{L^{2}} \ll 1$ small.\medskip

Theorem $\ref{t1.1}$ is proved using the $I$--method. The $I$--method was developed in \cite{colliander2002refined} and \cite{colliander2003sharp}. See also \cite{cazenave2003semilinear}, \cite{cazenave1990cauchy}, \cite{colliander2001global}

Let $I_{N, s}$ be the Fourier multiplier
\begin{equation}\label{1.9}
\widehat{I_{N, s} u}(k) = m_{N, s}(k) \hat{u}(k),
\end{equation}
where $m_{N, s}(k)$ is a smooth, even, monotone function satisfying $0 < m_{N, s}(k) \leq 1$, and
\begin{equation}\label{1.10}
m_{N, s}(|k|) = \left\{
     \begin{array}{lr}
       1 & : \qquad |k| \leq N, \\
       N^{1 - s} |k|^{s - 1} & : \qquad |k| > 2N. \\
     \end{array}
   \right.
\end{equation}

The operator $I_{N, s} : H^{s}(\mathbb{T}) \rightarrow H^{1}(\mathbb{T})$ and there exists a positive constant $C$ such that
\begin{equation}\label{1.11}
C^{-1} \| u \|_{H^{s}} \leq \| I_{N, s} u \|_{H^{1}} \leq C N^{1 - s} \| u \|_{H^{s}}.
\end{equation}
Thus,
\begin{equation}\label{1.12}
E(I_{N, s} u)(0) \lesssim N^{2(1 - s)} \| u(0) \|_{H^{s}}^{2}.
\end{equation}
We rescale the solution to $(\ref{1.1})$ so that $E(I_{N, s} u)(0) = \frac{1}{2}$ using the scaling in $(\ref{1.2})$. Note that if $u$ is a solution to $(\ref{1.1})$ on $\mathbb{T}$, then the scaling in $(\ref{1.2})$ maps a solution on $\mathbb{T}$ to a solution on $\lambda \mathbb{T}$.\medskip

The second step is to prove a long time Strichartz estimate for a solution to $(\ref{1.1})$ on $\lambda \mathbb{T}$ with $E(I_{N, s} u(t)) \leq 1$. It will be convenient to suppress the $N$ and $s$ in $I$ and just say $E(Iu(t)) \leq 1$ for the rest of the paper, and to understand that $I$ is a function of $N$ and $s$. Long time Strichartz estimates give good bounds on the solution $u$ to $(\ref{1.1})$ at frequencies $\geq N^{\delta}$ for some $\delta(s) > 0$ small. This is done using the bilinear estimates of \cite{planchon2009bilinear} adapted to the torus. Combining the long time Strichartz estimates with the modified energy $E(Iu(t))$ proves global well--posedness of $(\ref{1.1})$ when $s > \frac{1}{2}$.\medskip

To extend the global well--posedness result to $s > \frac{1}{3}$, we use a resonant decomposition of the modified energy $E(Iu(t))$. The resonant decomposition was utilized in conjunction with the $I$--method on $\mathbb{R}^{2}$ in \cite{collianderresonant}. Here we use the resonant decomposition on $\mathbb{T}$ used in \cite{li2011global}. Specifically,
\begin{equation}\label{1.13}
E(Iu) = M_{2}(u) + M_{6}(u) = M_{2}(u) + \bar{M}_{6}(u) + \tilde{M}_{6}(u),
\end{equation}
where $M_{2}$ is a $2$--linear Fourier operator and $M_{6}$ is a $6$--linear Fourier operator. Adding a correction term to $(\ref{1.13})$ gives an energy $\tilde{E}(u)$ that is both close to $E(Iu(t))$ when $E(Iu(t)) \leq 1$ and that is much more slowly varying in time, due to the presence of the long time Strichartz estimates. As a result, it is possible to prove global well--posedness when $s > \frac{1}{3}$.

\begin{remark}
Throughout this paper, $P_{\geq N}$, $P_{\leq N}$, and $P_{N}$ are the standard Littlewood--Paley projection operators. That is, $P_{\geq N}$ is a projection to frequencies $|\xi| \gtrsim N$, $P_{\leq N}$ is a projection to frequencies $|\xi| \lesssim N$, and $P_{N}$ is a projection to frequencies $|\xi| \sim N$.
\end{remark}

\section{A global well--posedness result on $\mathbb{R}$}
In this section we prove a global well--posedness result for $(\ref{1.3})$ on $\mathbb{R}$. Specifically,
\begin{theorem}\label{t2.1}
The initial value problem
\begin{equation}\label{2.1}
i u_{t} + u_{xx} - |u|^{4} u = 0, \qquad u : \mathbb{R} \times [0, T] \rightarrow \mathbb{C}, \qquad u(x, 0) = \phi \in H^{s}(\mathbb{R}),
\end{equation}
is globally well--posed for $\phi \in H^{s}(\mathbb{R})$ for any $s > 0$.
\end{theorem}
The proof of Theorem $\ref{t2.1}$ will utilize the I--method. Before proving this result, it is necessary to observe that the result in \cite{dodson2016global} is in fact stronger, there global well--posedness and scattering are proved for $(\ref{2.1})$ when $u_{0} \in L^{2}(\mathbb{R})$. However, there are a few reasons why it is useful to prove Theorem $\ref{t2.1}$ here. The first is that, as was discussed in the introduction, global well--posedness for $(\ref{1.1})$ almost certainly fails to hold for $\phi \in L^{2}(\mathbb{T})$. Scattering also almost certainly fails to hold for $\phi \in H^{s}(\mathbb{T})$, $s > 0$. So while Theorem $\ref{t2.1}$ is less than the optimal result on $\mathbb{R}$, if such a result could be proved on $\mathbb{T}$, it likely would be the optimal result in that case. The second reason is that Theorem $\ref{t2.1}$ will be proved using the I--method, rather than following the concentration compactness argument in \cite{dodson2016global} or \cite{herr2024strichartz}. Thus, Theorem $\ref{t2.1}$ provides a good warm--up for the I--method argument on $\mathbb{T}$.\medskip

\begin{proof}[Proof of Theorem $\ref{t2.1}$]
The proof of Theorem $\ref{t2.1}$ relies on long time Strichartz estimates and the bilinear interaction Morawetz estimate of \cite{planchon2009bilinear}.

\begin{theorem}[Bilinear estimate]\label{t2.2}
Let $u$ and $v$ be linear solutions to $i u_{t} + u_{xx} = 0$ and $i v_{t} + v_{xx} = 0$, respectively. Then,
\begin{equation}\label{2.2}
\int |\partial_{x}(u \bar{v})|^{2} dx dt \lesssim \| u \|_{L_{t}^{\infty} L_{x}^{2}}^{2} \| v \|_{L_{t}^{\infty} L_{x}^{2}} \| v \|_{L_{t}^{\infty} \dot{H}^{1}} +  \| v \|_{L_{t}^{\infty} L_{x}^{2}}^{2} \| u \|_{L_{t}^{\infty} L_{x}^{2}} \| u \|_{L_{t}^{\infty} \dot{H}^{1}}.
\end{equation}
\end{theorem}
\begin{proof}
The proof follows the proof of \cite{planchon2009bilinear}. Define the Morawetz potential
\begin{equation}\label{2.3}
M(t) = \int |u(t, y)|^{2} \frac{(x - y)}{|x - y|} Im[\bar{v} \partial_{x} v](t, x) dx dy + \int |v(t, y)|^{2} \frac{(x - y)}{|x - y|} Im[\bar{u} \partial_{x} u](t, x) dx dy.
\end{equation}
By direct computation,
\begin{equation}\label{2.4}
\frac{d}{dt} M(t) = 4 \int |\partial_{x}(u \bar{v})|^{2} dx.
\end{equation}
Then $(\ref{2.2})$ follows from the fundamental theorem of calculus.
\end{proof}

\begin{corollary}\label{c2.3}
If $u$ and $v$ are solutions to $i u_{t} + u_{xx} = 0$ and $i v_{t} + v_{xx} = 0$ respectively, and $u$ is supported on $|\xi| \sim N$, $v$ is supported on $|\xi| \sim M$, $M \ll N$, then
\begin{equation}\label{2.5}
\| u \bar{v} \|_{L_{t,x}^{2}(\mathbb{R} \times \mathbb{R})}^{2} \lesssim \frac{1}{N} \| u_{0} \|_{L^{2}}^{2} \| v_{0} \|_{L^{2}}^{2}.
\end{equation}
Since $\| uv \|_{L_{t,x}^{2}}^{2} = \| uv \bar{u} \bar{v} \|_{L_{t,x}^{1}}$,
\begin{equation}\label{2.6}
\| uv \|_{L_{t,x}^{2}(\mathbb{R} \times \mathbb{R})}^{2} \lesssim \frac{1}{N} \| u_{0} \|_{L^{2}}^{2} \| v_{0} \|_{L^{2}}^{2},
\end{equation}
as well.
\end{corollary}

Now, modifying $(\ref{1.1})$ to the continuous case, define the $I$--operator $I : H^{s}(\mathbb{R}) \rightarrow H^{1}(\mathbb{R})$ to be the Fourier multiplier
\begin{equation}\label{2.7}
\widehat{If}(\xi) = m(\xi) \hat{f}(\xi), \qquad m(\xi) = 1, \qquad \text{if} \qquad |\xi| \leq N, \qquad m(\xi) = |\xi|^{s - 1} N^{1 - s}, \qquad \text{if} \qquad |\xi| > 2N.
\end{equation}
By direct computation,
\begin{equation}\label{2.8}
\aligned
\frac{d}{dt} E(Iu(t)) = \langle \partial_{x} Iu_{t}, \partial_{x} Iu \rangle + \langle Iu_{t}, |Iu|^{4} (Iu) \rangle \\
= -\langle Iu_{t}, \Delta Iu \rangle + \langle Iu_{t}, |Iu|^{4} (Iu) \rangle \\
= \langle Iu_{t}, -\Delta Iu + I(|u|^{4} u) \rangle + \langle Iu_{t}, |Iu|^{4} (Iu) - I(|u|^{4} u) \rangle \\
= \langle Iu_{t}, |Iu|^{4} (Iu) - I(|u|^{4} u) \rangle \\ 
= \langle I(i \Delta u - i |u|^{4} u), |Iu|^{4} (Iu) - I(|u|^{4} u) \rangle.
\endaligned
\end{equation}
It is convenient to rescale the initial data so that $E(Iu_{0}) = \frac{1}{2}$ using the rescaling in $(\ref{1.2})$. Under the rescaling in $(\ref{1.2})$,
\begin{equation}\label{2.9}
\| u_{\lambda}(0) \|_{\dot{H}^{s}} = \lambda^{-s} \| u(0) \|_{\dot{H}^{s}},
\end{equation}
so therefore,
\begin{equation}\label{2.10}
E(Iu_{\lambda}(0)) \lesssim_{\| u_{0} \|_{L^{2}}} N^{2(1 - s)} \| u_{\lambda}(0) \|_{\dot{H}^{s}} = \lambda^{-2s} N^{2(1 - s)} \| u(0) \|_{\dot{H}^{s}}.
\end{equation}
Choosing $\lambda \sim_{\| u_{0} \|_{L^{2}}} N^{\frac{1 - s}{s}}$ gives $E(Iu_{\lambda}(0)) = \frac{1}{2}$.\medskip

Our aim is to show that for any $T < \infty$, we can choose $N(T)$ and $\lambda(T)$ such that $E(Iu_{\lambda}(t)) \leq 1$ for all $t \in [0, \lambda^{-2} T]$, which in turn implies that $(\ref{2.1})$ is globally well--posed on $[0, T]$. This would directly give global well--posedness, along with the polynomial bound
\begin{equation}\label{2.11}
\| u(t) \|_{H^{s}} \leq C(s, \| u_{0} \|_{H^{s}}) (1 + |t|)^{c(s)}, \qquad \text{for some} \qquad c(s) < \infty, \qquad \forall s > 0.
\end{equation}

To prove a bound on $E(Iu_{\lambda}(t))$, we will utilize some long time Strichartz estimates.
\begin{theorem}[Long time Strichartz estimates]\label{t2.4}
If $E(Iu_{\lambda}(t)) \leq 1$ on $[0, T_{0}] \subset [0, \lambda^{-2} T]$ then for any $\delta > 0$ and $s > 0$, there exists $N_{0}(s, \delta) < \infty$ such that for $N \geq N_{0}$,
\begin{equation}\label{2.12}
\| P_{\geq N^{\delta}} \partial_{x} Iu \|_{U_{\Delta}^{3}([0, T_{0}] \times \mathbb{R})} \lesssim 1,
\end{equation}
\begin{equation}\label{2.12.1}
\| (\partial_{x} I P_{\geq 8N^{\delta}} u)(P_{\leq N^{\delta}} u) \|_{L_{t,x}^{2}([0, T_{0}] \times \mathbb{R})} \lesssim 1,
\end{equation}
and
\begin{equation}\label{2.13}
\sup_{\| f \|_{L^{2}} = 1} \| (e^{it \partial_{xx}} P_{N} f)(P_{\leq N^{\delta}} u) \|_{L_{t,x}^{2}([0, T_{0}] \times \mathbb{R})} \lesssim N^{-1/2}.
\end{equation}
\end{theorem}

\begin{remark}
The $U_{\Delta}^{p}$ spaces are atomic spaces whose atoms are sums of solutions to the linear Schr{\"o}dinger equation, $e^{it \partial_{xx}} \phi_{k}$, supported on disjoint intervals. Specifically,
\begin{equation}
\sum_{k} 1_{[a_{k}, a_{k + 1})}(t) e^{it \partial_{xx}} \phi_{k}, \qquad \sum_{k} \| \phi_{k} \|_{L^{2}}^{p} = 1,
\end{equation}
where $\{ a_{k} \} \subset \mathbb{R}$ is an increasing sequence, and $1_{[a_{k}, a_{k + 1})}(t)$ are characteristic functions of intervals in $\mathbb{R}$. Thus, $U_{\Delta}^{p}$ spaces retain many properties of solutions to the linear Schr{\"o}dinger equation. For example, if $I \subset \mathbb{R}$ is an interval,
\begin{equation}
\| u \|_{L_{t,x}^{6}(I \times \mathbb{R})} \lesssim \| u \|_{U_{\Delta}^{3}(I \times \mathbb{R})}.
\end{equation}
\end{remark}

Assuming for a moment that the above theorem is true, we compute
\begin{equation}\label{2.14}
\int_{0}^{T_{0}} \frac{d}{dt} E(Iu(t)) dt = \int_{0}^{T_{0}} \langle I(i \Delta u - i |u|^{4} u), |Iu|^{4} (Iu) - I(|u|^{4} u) \rangle dt.
\end{equation}
Let $u_{l} = u_{\leq N/16}$ and $u_{h} = u_{> N/16}$, which means that $\hat{u}_{l}(\xi)$ has Fourier support on $|\xi| \leq N/8$.\medskip

\noindent \textbf{Case 1:} By Fourier support properties of $u_{l}$ and $u_{h}$,
\begin{equation}\label{2.15}
|Iu_{l}|^{4} (Iu_{l}) - I(|u_{l}|^{4} u_{l}) = 0.
\end{equation}

\noindent \textbf{Case 2:} Again by the Fourier support properties of $u_{l}$ and $u_{h}$,
\begin{equation}\label{2.16}
(Iu_{l})^{4} (Iu_{h}) - I(u_{l}^{4} u_{h}) = (Iu_{l})^{4} (I P_{> N/2} u) - I(u_{l}^{4} (P_{> N/2} u)).
\end{equation}
Now then, $|m(\xi) - m(\xi_{1} + \xi)| \lesssim \frac{|\xi_{1}|}{|\xi|} |m(\xi)|$, for $|\xi_{1}| \leq N/8$ and $|\xi| \geq N/2$. Therefore, integrating by parts and using the long time Strichartz estimates,
\begin{equation}\label{2.17}
\int_{0}^{T_{0}} \langle i I \Delta u, u_{l}^{4} (Iu_{h}) - I(u_{l}^{4} u_{h}) \rangle dt \lesssim \frac{1}{N} \| (\partial_{x} Iu_{h}) u_{l} \|_{L_{t,x}^{2}}^{2} \| \partial_{x} u_{l} \|_{L_{t,x}^{\infty}} \| u_{l} \|_{L_{t,x}^{\infty}} \lesssim N^{-1/2}.
\end{equation}
Also, using the long time Strichartz estimates,
\begin{equation}\label{2.18}
\int_{0}^{T_{0}} \langle i I(|u|^{4} u), u_{l}^{4} (Iu_{h}) - I(u_{l}^{4} u_{h}) \rangle dt \lesssim \| (P_{> N/2} u_{h}) u_{l} \|_{L_{t,x}^{2}}^{2} \| u_{l} \|_{L_{t,x}^{\infty}}^{6} + \| u_{h} \|_{L_{t,x}^{6}}^{6} \| u_{l} \|_{L_{t,x}^{\infty}}^{4} \lesssim N^{-2}.
\end{equation}

\noindent \textbf{Case 3:} In this case, integrating by parts,
\begin{equation}\label{2.19}
\aligned
\int_{0}^{T_{0}} \langle i I \Delta u, (Iu_{h})^{2} (Iu)^{3} - I(u_{h}^{2} u^{3}) \rangle dt \\
\lesssim \frac{1}{N} \| (\partial_{x} Iu_{h}) u_{l} \|_{L_{t,x}^{2}}^{2} \| u_{l} \|_{L_{t,x}^{\infty}} \| \partial_{x} u_{l} \|_{L_{t,x}^{\infty}} + \| |\partial_{x} Iu_{h}|^{2} |u_{h}|^{4} \|_{L_{t,x}^{1}} \\ + \frac{1}{N} \| (\partial_{x} Iu_{h}) u_{l} \|_{L_{t,x}^{2}}^{3/2} \| \partial_{x} Iu_{h} \|_{L_{t,x}^{6}}^{3/2} \| u_{l} \|_{L_{t,x}^{\infty}}^{3/2} \lesssim \frac{1}{N^{1/2}}.
\endaligned
\end{equation}

\noindent \textbf{Case 4:} For
\begin{equation}\label{2.20}
\int_{0}^{T_{0}} \langle i I(|u|^{4} u), (Iu_{l})^{4} (Iu_{h}) - I(u_{l}^{4} u_{h}) \rangle dt,
\end{equation}
split $I(|u|^{4} u) = I(O(u_{h}^{3} u^{2}) + O(u_{l}^{3} u^{2}))$. First,
\begin{equation}\label{2.21}
\aligned
\int_{0}^{T_{0}} \langle i I(u_{h}^{3} u^{2}), (Iu_{l})^{4} (Iu_{h}) - I(u_{l}^{4} u_{h}) \rangle dt \\
\lesssim \| (u_{h}) u_{l} \|_{L_{t,x}^{2}} \| u_{h} \|_{L_{t,x}^{6}}^{3} \| u_{l} \|_{L_{t,x}^{\infty}}^{5} + \| u_{h} \|_{L_{t,x}^{6}}^{6} \| u_{l} \|_{L_{t,x}^{\infty}}^{4} \lesssim \frac{1}{N^{4}}.
\endaligned
\end{equation}
Meanwhile, following $(\ref{2.17})$, at least one term in $I(|u|^{4} u)$ must be at high frequency. Therefore,
\begin{equation}\label{2.22}
\aligned
\int_{0}^{T_{0}} \langle i I(u_{l}^{3} u^{2}), (Iu_{l})^{4} (Iu_{h}) - I(u_{l}^{4} u_{h}) \rangle dt \\
\lesssim \| (Iu_{h}) u_{l} \|_{L_{t,x}^{2}}^{2} \| Iu \|_{L_{t,x}^{\infty}}^{6} + \| (u_{h}) u_{l} \|_{L_{t,x}^{2}}^{2} \| Iu_{h} \|_{L_{t, x}^{\infty}}  \| u_{l} \|_{L_{t,x}^{\infty}}^{5} \lesssim \frac{1}{N^{2}}.
\endaligned
\end{equation}

\noindent \textbf{Case 5:}
\begin{equation}\label{2.23}
\aligned
\int_{0}^{T_{0}} \langle iI(|u|^{4} u), I(10 u_{h}^{2} u_{l}^{3} + 10 u_{h}^{3} u_{l}^{2} + 5 u_{h}^{4} u_{l} + u_{h}^{5}) \rangle dt \\
= \int_{0}^{T_{0}} \langle iI(|u_{l}|^{4} u_{l} + 5 u_{l}^{4} u_{h}), I(10 u_{h}^{2} u_{l}^{3} + 10 u_{h}^{3} u_{l}^{2} + 5 u_{h}^{4} u_{l} + u_{h}^{5}) \rangle dt \\
\lesssim \| u_{l}^{4} u_{h}^{6} \|_{L_{t,x}^{1}} + \| u_{l}^{6} u_{h}^{4} \|_{L_{t,x}^{1}} \lesssim \| u_{h} \|_{L_{t,x}^{6}}^{6} \| u_{l} \|_{L_{t,x}^{\infty}}^{4} + \| u_{l} \|_{L_{t}^{\infty} L_{x}^{6}}^{5} \| u_{h} u_{l} \|_{L_{t,x}^{2}} \| u_{h} \|_{L_{t,x}^{6}}^{3} \lesssim \frac{1}{N^{4}}.
\endaligned
\end{equation}

Therefore, for any $t \in [0, T_{0}]$, $E(Iu(t)) \leq \frac{1}{2} + O(N^{-1/2}) \leq \frac{3}{4}$. Therefore, making a bootstrap argument, $E(Iu(t)) \leq 1$ on $[0, \lambda^{2} T]$, which proves Theorem $\ref{t2.1}$.
\end{proof}

It only remains to prove the long time Strichartz estimate, Theorem $\ref{t2.4}$.
\begin{proof}[Proof of Theorem $\ref{t2.4}$]
The estimate in $(\ref{2.13})$ is proved using the bilinear Morawetz estimate in \cite{planchon2009bilinear}, see Theorem $\ref{t2.2}$. Since
\begin{equation}\label{2.24}
i u_{t} + u_{xx} = \mathcal N,
\end{equation}
let
\begin{equation}\label{2.25}
\aligned
M(t) = \int |P_{\geq M} Iu(t, y)|^{2} \frac{(x - y)}{|x - y|} Im[\overline{P_{\leq M/8} u} \partial_{x} P_{\leq M/8} u](t, x) dx dy \\
+ \int |P_{\leq M/8} u(t, y)|^{2} \frac{(x - y)}{|x - y|} Im[\overline{P_{\geq M} Iu} \partial_{x} P_{\geq M} u](t, x) dx dy,
\endaligned
\end{equation}
for some $M \leq N/8$. Then following the proof of Theorem $\ref{t2.2}$,
\begin{equation}\label{2.26}
\aligned
4 \int_{0}^{T_{0}} \int |\partial_{x}(\overline{P_{\geq M} Iu}(t, x) P_{\leq M/8} u(t, x))|^{2} dx dt \lesssim \| Iu \|_{L_{t}^{\infty} H_{x}^{1}([0, T_{0}] \times \mathbb{R})} \| u \|_{L_{t}^{\infty} L_{x}^{2}([0, T_{0}] \times \mathbb{R})}^{3} \\
- 2 \int_{0}^{T_{0}} \int Im[\overline{P_{\geq M} I u(t, y)} P_{\geq M} I \mathcal N] \frac{(x - y)}{|x - y|} Im[\overline{P_{\leq M/8} u} \partial_{x} P_{\leq M/8} u](t, x) dx dy dt \\
- 2 \int_{0}^{T_{0}} \int Im[\overline{P_{\leq M/8} I u(t, y)} P_{\leq M/8} I \mathcal N] \frac{(x - y)}{|x - y|} Im[\overline{P_{\geq M} Iu} \partial_{x} P_{\geq M} Iu](t, x) dx dy dt \\
+ \int_{0}^{T_{0}} \int |P_{\geq M} Iu(t, y)|^{2} \frac{(x - y)}{|x - y|} Re[\overline{P_{\leq M/8} u} \partial_{x} P_{\leq M/8} \mathcal N](t, x) dx dy \\
- \int_{0}^{T_{0}} \int |P_{\geq M} Iu(t, y)|^{2} \frac{(x - y)}{|x - y|} Re[\overline{P_{\leq M/8} \mathcal N} \partial_{x} P_{\leq M/8} u](t, x) dx dy \\
+ \int_{0}^{T_{0}} \int |P_{\leq M/8} u(t, y)|^{2} \frac{(x - y)}{|x - y|} Re[\overline{P_{\geq M} Iu} \partial_{x} P_{\geq M} I \mathcal N](t, x) dx dy \\
- \int_{0}^{T_{0}} \int |P_{\leq M/8} u(t, y)|^{2} \frac{(x - y)}{|x - y|} Re[\overline{P_{\geq M} I \mathcal N} \partial_{x} P_{\geq M} Iu](t, x) dx dy = O(1) + \mathcal E.
\endaligned
\end{equation}
Now let $u_{l} = u_{\leq M/64}$, $u_{h} = u_{\geq M/64}$, and decompose
\begin{equation}\label{2.27}
\mathcal N = |u_{l}|^{4} u_{l} + (3 |u_{l}|^{4} u_{h} + 2 |u_{l}|^{2} u_{l}^{2} \overline{u_{h}}) + O(u_{h}^{2} u^{3}) = \mathcal N_{0} + \mathcal N_{1} + \mathcal N_{2}.
\end{equation}
Then $\mathcal E = \mathcal E_{0} + \mathcal E_{1} + \mathcal E_{2}$, where $\mathcal E_{0}$ is the error terms in $(\ref{2.26})$ with $\mathcal N$ replaced by $\mathcal N_{0}$ and the corresponding $u$ replaced by $u_{l}$. Likewise, $\mathcal E_{1}$ is the error term in $(\ref{2.26})$ with $\mathcal N$ replaced by $\mathcal N_{1}$ and the corresponding $u$ replaced by $u_{l}$ or with $\mathcal N$ replaced by $\mathcal N_{0}$ and the corresponding $u$ replaced by $u_{h}$; and $\mathcal E_{2}$ is the error term with $\mathcal N$ replaced by $\mathcal N_{2}$ or with $\mathcal N$ replaced by $\mathcal N_{1}$ and the corresponding $u$ replaced by $u_{h}$.

\begin{claim}\label{claim1}
\begin{equation}\label{2.28}
\mathcal E_{0} \leq 0.
\end{equation}
\end{claim}
\begin{proof}[Proof of claim]
This follows by integrating by parts. By Fourier support properties of $u_{l}$, $P_{\geq M} I \mathcal N_{0} = 0$, so $Im[\overline{P_{\geq M} I u_{l}(t, y)} P_{\geq M} I \mathcal N_{0}] = 0$. Next, since $P_{\leq M/8} (|u_{l}|^{4} u_{l}) = |u_{l}|^{4} u_{l}$ and $P_{\leq M/8} u_{l} = u_{l}$, $Im[\overline{P_{\leq M/8} I u_{l} (t, y)} P_{\leq M/8} I \mathcal N_{0}] = 0$. Therefore, integrating by parts,
\begin{equation}\label{2.29}
\mathcal E_{0} = -\int_{J} \int |P_{\geq M} Iu(t,x)|^{2} |u_{l}(t, x)|^{6} dx dt \leq 0.
\end{equation}
\end{proof}

\begin{claim}\label{claim2}
If $E(Iu(t)) \leq 1$ on $J$,
\begin{equation}\label{2.30}
\aligned
\mathcal E_{2} \lesssim \frac{1}{M} \| (\partial_{x} P_{\geq M/64} Iu)(P_{\leq M/512} u) \|_{L_{t,x}^{2}} \| (P_{\geq M/64} Iu)(P_{\leq M/512} u) \|_{L_{t,x}^{2}} \\ + \frac{1}{M} \| |\partial_{x} P_{\geq M/64} Iu| |P_{\geq M/64} Iu| |P_{\geq M/512} u|^{4} \|_{L_{t,x}^{1}}.
\endaligned
\end{equation}
\end{claim}
\begin{proof}[Proof of claim]
By H{\"o}lder's inequality,
\begin{equation}\label{2.31}
\| u_{h}^{2} u^{4} \|_{L_{t,x}^{1}} \lesssim \| (u_{h})(u_{\leq M/512}) \|_{L_{t,x}^{2}}^{2} \| u_{\leq M/512} \|_{L_{t,x}^{\infty}}^{2} + \| u_{h} \|_{L_{t,x}^{6}}^{2} \| u_{\geq M/512} \|_{L_{t,x}^{6}}^{4}.
\end{equation}
Since $M \leq N/8$, $E(Iu(t)) \leq 1$ implies $\| u_{\leq M/512} \|_{L^{\infty}} \lesssim 1$. Since $|\mathcal N_{2} u| + |\mathcal N_{1} u_{h}| \lesssim |u_{h}|^{2} |u|^{4}$, plugging $(\ref{2.31})$ combined with Bernstein's inequality into $(\ref{2.26})$ implies $(\ref{2.30})$.
\end{proof}

\begin{claim}\label{claim3}
If $E(Iu(t)) \leq 1$ on $J$,
\begin{equation}\label{2.32}
\aligned
\mathcal E_{1} \lesssim \frac{1}{M} \| (\partial_{x} P_{\geq M/64} Iu)(P_{\leq M/512} u) \|_{L_{t,x}^{2}} \| (P_{\geq M/64} Iu)(P_{\leq M/512} u) \|_{L_{t,x}^{2}} \\ + \frac{1}{M} \| |\partial_{x} P_{\geq M/64} Iu| |P_{\geq M/64} Iu| |P_{\geq M/512} u|^{4} \|_{L_{t,x}^{1}}.
\endaligned
\end{equation}
\end{claim}
\begin{proof}[Proof of claim]
By Fourier support properties of $u_{l}$ and $\mathcal N_{0}$,
\begin{equation}\label{2.33}
Im[\overline{P_{\geq M} Iu_{l}} P_{\geq M} I \mathcal N_{1}] + Im[\overline{P_{\geq M} I u_{h}} P_{\geq M} I \mathcal N_{0}] = 0,
\end{equation}
and
\begin{equation}\label{2.34}
Re[\overline{P_{\geq M} Iu_{h}} \partial_{x} P_{\geq M} I \mathcal N_{0}] + Re[\overline{P_{\geq M} I u_{l}} \partial_{x} P_{\geq M} I \mathcal N_{1}] = 0.
\end{equation}
Next, decomposing $u_{l} = u_{\geq M/512} + u_{\leq M/512}$,
\begin{equation}\label{2.35}
Im[\overline{P_{\leq M/8} Iu_{l}} P_{\leq M/8} I \mathcal N_{1}] = O(u_{h} u_{\geq M/512} u^{4}) + Im[\overline{P_{\leq M/8} Iu_{\leq M/512}} P_{\leq M/8} I(u_{h} u_{\leq M/512}^{4})].
\end{equation}
As in $(\ref{2.31})$,
\begin{equation}\label{2.36}
\| u_{h} (u_{\geq M/512}) u^{4} \|_{L_{t,x}^{1}} \lesssim \| (u_{h}) (u_{\leq M/512}) \|_{L_{t,x}^{2}} \| u_{\geq M/512} (u_{\leq M/8^{4}}) \|_{L_{t,x}^{2}} \| u_{\leq M/512} \|_{L_{t,x}^{\infty}}^{2} + \| u_{\geq M/512} \|_{L_{t,x}^{6}}^{6}.
\end{equation}
Furthermore,
\begin{equation}\label{2.37}
Im[\overline{P_{\leq M/8} Iu_{\leq M/512}} P_{\leq M/8} I(u_{h} u_{\leq M/512}^{4})],
\end{equation}
is supported on $|\xi| \gtrsim \frac{M}{512}$. By Littlewood--Paley arguments, $P_{\geq \frac{M}{512}} (\frac{x}{|x|})$ is a function that is $\lesssim 1$ and is rapidly decreasing for $|x| > \frac{1}{M}$. That is, for any $N > 0$, $P_{\geq \frac{M}{512}}(\frac{x}{|x|}) \lesssim (1 + M |x|)^{-N}$. Therefore, since $\| (1 + M |x|)^{-N} \|_{L^{1}} \lesssim \frac{1}{M}$ and $\| (1 + M |x|)^{-N} \|_{L^{\infty}} \lesssim 1$,
\begin{equation}\label{2.38}
\aligned
\int_{0}^{T_{0}} \int \int Im[\overline{P_{\leq M/8} Iu_{\leq M/512}} P_{\leq M/8} I(u_{h} u_{\leq M/512}^{4})] \frac{(x - y)}{|x - y|} Im[\overline{P_{\geq M} Iu} \partial_{x} P_{\geq M} Iu](t, x) dx dy dt \\
\lesssim \frac{1}{M} \| (u_{h}) u_{\leq M/512} u \|_{L_{t,x}^{2}} \| (P_{\geq M} \partial_{x} Iu) u_{\leq M/512} \|_{L_{t,x}^{2}} \| Iu \|_{L_{t,x}^{\infty}}^{4}.
\endaligned
\end{equation}
The computations with $Im[\overline{P_{\leq M/8} Iu_{h}} P_{\leq M/8} I \mathcal N_{0}]$ are similar, as are the computations with
\begin{equation}\label{2.39}
\aligned
\int_{0}^{T_{0}} \int |P_{\geq M} Iu(t, y)|^{2} \frac{(x - y)}{|x - y|} Re[\overline{P_{\leq M/8} u_{h}} \partial_{x} P_{\leq M/8} \mathcal N_{0}](t, x) dx dy \\
+ \int_{0}^{T_{0}} \int |P_{\geq M} Iu(t, y)|^{2} \frac{(x - y)}{|x - y|} Re[\overline{P_{\leq M/8} u_{l}} \partial_{x} P_{\leq M/8} \mathcal N_{1}](t, x) dx dy \\
- \int_{0}^{T_{0}} \int |P_{\geq M} Iu(t, y)|^{2} \frac{(x - y)}{|x - y|} Re[\overline{P_{\leq M/8} \mathcal N_{0}} \partial_{x} P_{\leq M/8} u_{h}](t, x) dx dy \\
- \int_{0}^{T_{0}} \int |P_{\geq M} Iu(t, y)|^{2} \frac{(x - y)}{|x - y|} Re[\overline{P_{\leq M/8} \mathcal N_{1}} \partial_{x} P_{\leq M/8} u_{l}](t, x) dx dy \\
+ \int_{0}^{T_{0}} \int |P_{\leq M/8} u(t, y)|^{2} \frac{(x - y)}{|x - y|} Re[\overline{P_{\geq M} Iu_{h}} \partial_{x} P_{\geq M} I \mathcal N_{0}](t, x) dx dy \\
+ \int_{0}^{T_{0}} \int |P_{\leq M/8} u(t, y)|^{2} \frac{(x - y)}{|x - y|} Re[\overline{P_{\geq M} Iu_{l}} \partial_{x} P_{\geq M} I \mathcal N_{1}](t, x) dx dy.
\endaligned
\end{equation}

Now prove Theorem $\ref{t2.4}$ by induction on frequency, starting with frequency $M_{0} = N^{\delta/2}$. Using the fact that $E(Iu(t)) \leq 1$ on $[0, T_{0}]$, local well--posedness implies that for any interval $|J| = 1$, $J \subset [0, T_{0}]$, $(\ref{2.12})$ and $(\ref{2.13})$ hold when $|J| = 1$, and
\begin{equation}\label{2.39.1}
\| (\partial_{x} Iu_{\geq M_{0}}) (u_{\leq M_{0}/8}) \|_{L_{t,x}^{2}(J)} \lesssim 1.
\end{equation}

Now let $J_{k}^{(0)} = [k, k + 1] \subset [0, T_{0}]$, and for any $j > 0$, let $J_{k}^{(j)} \subset [0, T_{0}]$ be unions of $\sim N^{\delta/10}$ consecutive $J_{k'}^{(j - 1)} \subset [0, T_{0}]$. Furthermore let $M_{0} = N^{\delta/2}$ and let $M_{j} = 8^{4j} M_{0}$. If for some $j > 0$, $(\ref{2.12})$ and $(\ref{2.13})$ hold for $J_{k}^{(j - 1)} \subset [0, T_{0}]$ and , $M = M_{j - 1}$, then $(\ref{2.28})$, $(\ref{2.30})$, $(\ref{2.32})$ imply
\begin{equation}\label{2.40}
\| (\partial_{x} Iu_{\geq M_{j}})(u_{\leq M_{j}/8}) \|_{L_{t,x}^{2}(J_{k}^{(j)})}^{2} \lesssim M_{j}^{-1} N^{\delta/10} + 1,
\end{equation}
which implies $(\ref{2.39})$ holds at level $M_{j}$.
\begin{equation}\label{2.41}
\| (\partial_{x} Iu_{\geq M_{j}})(u_{\leq M_{j}/8}) \|_{L_{t,x}^{2}(J_{k}^{(j)})}^{2} \lesssim 1.
\end{equation}
\begin{remark}
Observe that
\begin{equation}
\| \partial_{x}((I P_{\geq M_{j}} u)(P_{\leq M_{j}/8} u)) \|_{L_{t,x}^{2}} \sim \| (\partial_{x} Iu_{\geq M_{j}})(u_{\leq M_{j}/8}) \|_{L_{t,x}^{2}}.
\end{equation}
\end{remark}

To obtain $(\ref{2.13})$ at level $M_{j}$, note that in $(\ref{2.26})$, $u_{\geq M}$ could easily be replaced by a linear solution $e^{it \partial_{xx}} P_{2^{k} M} f$, for any $k \geq 0$. Observe that in this case,
\begin{equation}
\| e^{it \partial_{xx}} P_{2^{k} M} f \|_{\dot{H}^{1}} \sim 2^{k} M \| P_{2^{k} M} f \|_{L^{2}},
\end{equation}
and
\begin{equation}
\| \partial_{x}((e^{it \partial_{xx}} P_{2^{k} M} f)(P_{\leq M/8} u)) \|_{L_{t,x}^{2}}^{2} \sim 2^{2k} M^{2} \| ((e^{it \partial_{xx}} P_{2^{k} M} f)(P_{\leq M/8} u)) \|_{L_{t,x}^{2}}^{2},
\end{equation}
so we can do some algebra and sum over $k \geq 0$. In this case, the computations for $\mathcal E$ are simplified since $\mathcal N = 0$ for the $e^{it \partial_{xx}} f$ terms, which would give $(\ref{2.13})$ at the level $M_{j}$. 

Finally, using \cite{hadac2009well}, to obtain $(\ref{2.12})$ it is enough to obtain a bound on
\begin{equation}\label{2.42}
\| v (P_{\geq M_{j}} |u|^{4} u) \|_{L_{t,x}^{1}},
\end{equation}
when $\| v \|_{V_{\Delta}^{3/2}} = 1$ and $\hat{v}$ is supported on $|\xi| \geq M_{j}$. Since $V_{\Delta}^{3/2} \subset U_{\Delta}^{2}$ (see \cite{hadac2009well}), $\| v \|_{U_{\Delta}^{2}} \leq 1$ as well. Decomposing $v$ into $U_{\Delta}^{2}$ atoms and using the fractional product rule, it is enough to bound
\begin{equation}\label{2.43}
\aligned
\| (e^{it \partial_{xx}} P_{\geq M_{j}} f) \partial_{x} I P_{\geq M_{j}}(|u|^{4} u) \|_{L_{t,x}^{1}} \\ \lesssim \| (e^{it \partial_{xx}} P_{\geq M_{j}} f)(P_{\leq M_{j}/8} u) \|_{L_{t,x}^{2}} \| (\partial_{x} I u_{\geq M_{j}/8})(u_{\leq M_{j}/64}) \|_{L_{t,x}^{2}} \| u_{\leq M_{j}/64} \|_{L_{t,x}^{\infty}}^{2} \\ + \| e^{it \partial_{xx}} P_{\geq M_{j}} f \|_{L_{t,x}^{6}} \| |\partial_{x} I u_{\geq M_{j}/8}| |P_{\geq M_{j}/64} u|^{4} \|_{L_{t,x}^{1}} \lesssim M_{j}^{-1}.
\endaligned
\end{equation} 
The last bound uses the estimates at level $M_{j - 1}$.\medskip

There exists $c > 0$ such that we can take $c \delta \log(N)$ steps between $M_{0} = N^{\delta/2}$ and $N^{\delta}$ with $M_{j} = 8^{4j} M_{0}$. Since $T_{0} \leq T N^{\frac{2(1 - s)}{s}}$, for $N(s, T, \delta)$ sufficiently large, $N^{c \delta \log N} \gg T N^{\frac{2(1 - s)}{s}}$, and therefore, for some $j \leq c \delta \log N$, $J_{1}^{(j)} = [0, T_{0}]$, which proves Theorem $\ref{t2.4}$.

\end{proof}
This proves Theorem $\ref{t2.1}$.

\end{proof}

\section{The long time Strichartz estimate adapted to the torus}
Now turn to the nonlinear Schr{\"o}dinger equation on $\mathbb{T}$. The proof of Theorem $\ref{t1.1}$ will utilize a long time Strichartz estimate adapted to $\mathbb{T}$ and the modified energy utilized in \cite{li2011global}. For expository reasons, it will be useful to first prove a global well--posedness result for $s > \frac{1}{2}$ using only the long time Strichartz estimates adapted to the torus. This will then be extended to $s > \frac{1}{3}$ in the next section using the modified energy of \cite{li2011global}.
\begin{theorem}\label{t3.1}
The cubic NLS
\begin{equation}\label{3.1}
i \partial_{t} u + \partial_{xx} u - |u|^{4} u = 0, \qquad u(x, 0) = \phi(x) \in H^{s}(\mathbb{T}), \qquad u : \mathbb{T} \times [0, T] \rightarrow \mathbb{C},
\end{equation}
is globally well--posed for $\phi \in H^{s}(\mathbb{T})$, $s > 1/2$.
\end{theorem}
This result is not as strong as the global well--posedness result of \cite{li2011global}. However, here we do not need to use a modified energy. Instead, we will use a long time Strichartz estimate.\medskip

Adapting the scaling in $(\ref{1.2})$ to the torus,
\begin{equation}\label{3.2}
u_{\lambda}(x, t) = \lambda^{-1/2} u(\lambda^{-1} x, \lambda^{-2} t),
\end{equation}
solves
\begin{equation}\label{3.3}
i \partial_{t} u + \partial_{xx} u - |u|^{4} u = 0,
\end{equation}
on $\lambda \mathbb{T} = \mathbb{R} / \lambda \mathbb{Z}$. Let $I : H^{s}(\lambda \mathbb{T}) \rightarrow H^{1}(\lambda \mathbb{T})$ be the Fourier multiplier given by $(\ref{1.9})$ and $(\ref{1.10})$.
%\begin{equation}\label{3.4}
%m(|n|) = \left\{
%     \begin{array}{lr}
%       1 & : \qquad |n| \leq N, \\
%       N^{1 - s} |n|^{s - 1} & : \qquad |n| > 2N. \\
%     \end{array}
%   \right.
%\end{equation}
Then,
\begin{equation}\label{3.5}
\aligned
E(Iu(t)) = \frac{1}{2} \| Iu \|_{\dot{H}^{1}}^{2} + \frac{1}{6} \| Iu \|_{L^{6}}^{6} \leq \frac{1}{2} \| Iu \|_{\dot{H}^{1}}^{2} + \frac{1}{6} \| Iu \|_{L^{2}}^{2} \| Iu \|_{L^{\infty}}^{4} \lesssim \| Iu \|_{\dot{H}^{1}}^{2} + \| Iu \|_{\dot{H}^{1}}^{2} \| Iu \|_{L^{2}}^{4} + \frac{1}{\lambda^{2}} \| Iu \|_{L^{2}}^{6} \\ \lesssim N^{2(1 - s)} \| (u_{0})_{\lambda} \|_{\dot{H}^{s}}^{2} + \frac{1}{\lambda^{2}} = \frac{N^{2(1 - s)}}{\lambda^{2s}} \| u_{0} \|_{\dot{H}^{s}}^{2} + \frac{1}{\lambda^{2}}.
\endaligned
\end{equation}
Therefore, taking $\lambda \sim N^{\frac{1 - s}{s}}$, $E(Iu_{\lambda}(0)) = \frac{1}{2}$.\medskip
%\begin{remark}
%The mass $\int_{\lambda \mathbb{T}} |u(t, x)|^{2} dx$ and energy $ \int_{\lambda \mathbb{T}} [\frac{1}{2} |u_{x}(t,x)|^{2} + \frac{1}{6} |u(t,x)|^{6}] dx$ are both conserved quantities.
%\end{remark}

Now prove a long time Strichartz estimate adapted to the torus.
\begin{proposition}[Long time Strichartz estimate]\label{p3.2}
For any $\delta > 0$, $s > \frac{1}{3}$, there exists $N \geq N_{0}(s, \delta)$ such that if $J \subset \mathbb{R}$ is an interval, $|J| \leq \lambda$, $E(Iu_{\lambda}(t)) \leq 1$ on $J$, then
\begin{equation}\label{3.6}
\aligned
\| ( \partial_{x} I P_{\geq 8 N^{\delta}} u)(P_{\leq N^{\delta}} u) \|_{L_{t,x}^{2}(J \times \lambda \mathbb{T})} \lesssim 1, \\
\| \partial_{x} I P_{> N^{\delta}} u \|_{U_{\Delta}^{3}(J \times \lambda \mathbb{T})} \lesssim 1,
\endaligned
\end{equation}
for any $k > 0$,
\begin{equation}\label{3.6.1}
\| (\partial_{x} I P_{\geq 8 2^{k} N^{\delta}} u)(P_{2^{k} N^{\delta}} u) \|_{L_{t,x}^{2}(J \times \lambda \mathbb{T})} \lesssim \frac{1}{m(2^{k} N^{\delta}) 2^{k} N^{\delta}},
\end{equation}
and
\begin{equation}\label{3.7}
\sup_{\| f \|_{L^{2}} = 1} \| (P_{> N} e^{it \partial_{xx}} f)(P_{\leq N^{\delta}} u) \|_{L_{t,x}^{2}(J \times \lambda \mathbb{T})} \lesssim 1.
\end{equation}
\end{proposition}
\begin{remark}
The additional term in $(\ref{3.6.1})$ is needed due to the fact that the Strichartz estimates in \cite{bourgain1993fourier} have a loss. The improved Strichartz estimates in \cite{skouloudis2026strichartz} are not enough to make the term $(\ref{3.6.1})$ unnecessary.
\end{remark}
\begin{proof}
The function $\frac{(x - y)}{|x - y|}$ is not well--defined on the torus, since $x, y \in \mathbb{R} / \lambda \mathbb{Z}$. Therefore, to define an interaction Morawetz potential, we treat $u$ as a function on $\mathbb{R}$ that is periodic of period $\lambda$. A consequence of this is that such a function does not lie in $L^{2}(\mathbb{R})$, so let $\chi \in C_{0}^{\infty}(\mathbb{R})$ be a set with compact support, $\chi(x) = 1$ on $0 \leq x \leq \lambda$, and let
\begin{equation}\label{3.8}
\aligned
M(t) = \int \int |P_{> 8M} Iu|^{2} \chi(y) \frac{(x - y)}{|x - y|} \chi(x) Im[\overline{P_{\leq M} u} \partial_{x} P_{\leq M} u] dx dy \\
+ \int \int |P_{\leq M} u|^{2} \chi(y) \frac{(x - y)}{|x - y|} \chi(x) Im[\overline{P_{> 8M} Iu} \partial_{x} P_{> 8M} Iu] dx dy.
\endaligned
\end{equation}
Taking the time derivative of $M(t)$ and integrating in time, then as in $(\ref{2.26})$,
\begin{equation}\label{3.9}
\aligned
4 \int_{0}^{T_{0}} \int \chi(x) |\partial_{x}(\overline{I P_{\geq M} u}(t, x) P_{\leq M/8} u(t, x))|^{2} dx dt \lesssim \| Iu \|_{L_{t}^{\infty} H_{x}^{1}([0, T_{0}] \times \mathbb{T})} \| u \|_{L_{t}^{\infty} L_{x}^{2}([0, T_{0}] \times \mathbb{T})}^{3} \\
- 2 \int_{0}^{T_{0}} \int Im[\overline{P_{\geq M} I u(t, y)} P_{\geq M} I \mathcal N] \chi(y) \frac{(x - y)}{|x - y|} \chi(x) Im[\overline{P_{\leq M/8} u} \partial_{x} P_{\leq M/8} u](t, x) dx dy dt \\
- 2 \int_{0}^{T_{0}} \int Im[\overline{P_{\leq M/8} I u(t, y)} P_{\leq M/8} I \mathcal N] \chi(y) \frac{(x - y)}{|x - y|} \chi(x) Im[\overline{P_{\geq M} Iu} \partial_{x} P_{\geq M} Iu](t, x) dx dy dt \\
+ \int_{0}^{T_{0}} \int |P_{\geq M} Iu(t, y)|^{2} \chi(y) \frac{(x - y)}{|x - y|} \chi(x) Re[\overline{P_{\leq M/8} u} \partial_{x} P_{\leq M/8} \mathcal N](t, x) dx dy \\
- \int_{0}^{T_{0}} \int |P_{\geq M} Iu(t, y)|^{2} \chi(y) \frac{(x - y)}{|x - y|} \chi(x) Re[\overline{P_{\leq M/8} \mathcal N} \partial_{x} P_{\leq M/8} u](t, x) dx dy \\
+ \int_{0}^{T_{0}} \int |P_{\leq M/8} u(t, y)|^{2} \chi(y) \frac{(x - y)}{|x - y|} \chi(x) Re[\overline{P_{\geq M} Iu} \partial_{x} P_{\geq M} I \mathcal N](t, x) dx dy \\
- \int_{0}^{T_{0}} \int |P_{\leq M/8} u(t, y)|^{2} \chi(y) \frac{(x - y)}{|x - y|} \chi(x) Re[\overline{P_{\geq M} I \mathcal N} \partial_{x} P_{\geq M} Iu](t, x) dx dy \\
+ \int_{0}^{T_{0}} \frac{1}{\lambda} O(\| Iu \|_{H^{1}(\mathbb{T})}^{2} \| u \|_{L^{2}(\mathbb{T})}^{2}) dt + \int_{0}^{T_{0}} \frac{1}{\lambda^{3}} O(\| u \|_{L^{2}(\mathbb{T})}^{4}) dt.
\endaligned
\end{equation}
The last term in $(\ref{3.9})$ arises from integrating by parts, when the derivative hits $\chi$. When $|T_{0}| \leq \lambda$, $E(Iu(t)) \leq 1$ and $\| u \|_{L^{2}} < \infty$,
\begin{equation}\label{3.10}
\int_{0}^{T_{0}} \frac{1}{\lambda} O(\| Iu \|_{H^{1}}^{2} \| u \|_{L^{2}}^{2}) dt + \int_{0}^{T_{0}} \frac{1}{\lambda^{3}} O(\| u \|_{L^{2}}^{4}) dt \lesssim 1.
\end{equation}
\begin{remark}
In $(\ref{3.9})$, $(\ref{3.10})$, and elsewhere in the section, $\| Iu \|_{L^{2}}$ or $\| Iu \|_{H^{1}}$ refers to the $L^{2}$ or $H^{1}$ norm on the torus, not on $\mathbb{R}$, since $u$ is a periodic function. It is straightforward to verify that if $f$ is a periodic function of period $\lambda$, and $\chi$ is the cut--off function in $(\ref{3.8})$,
\begin{equation}\label{3.11}
\int_{\mathbb{R}} \chi(x) |f(x)|^{p} dx \sim \| f \|_{L^{p}(\lambda \mathbb{T})}^{p}.
\end{equation}
\end{remark}

Now turn to the terms in $(\ref{3.9})$ with $\mathcal N$. The estimates in Claim $\ref{claim1}$ and Claim $\ref{claim2}$ for the Schr{\"o}dinger equation on $\mathbb{R}$ can be applied verbatim to give estimates on $\lambda \mathbb{T}$, as can the estimates in $(\ref{2.36})$ for Claim $\ref{claim3}$, except for one technical issue that will be discussed momentarily. The estimate of the analog of $(\ref{2.37})$ is slightly different, since in this case we need to estimate
\begin{equation}\label{3.11}
P_{\geq \frac{M}{8^{4}}} (\chi(x) \frac{(x - y)}{|x - y|}),
\end{equation}
whose kernel satisfies the bounds
\begin{equation}\label{3.12}
K(x, z) \lesssim_{N} M (1 + M |x - z|)^{-N} + O(\frac{1}{\lambda}) [\int_{0}^{1} |\chi'(z + (x - z)t)| dt].
\end{equation}
Comparing the contribution of the additional term in $(\ref{3.12})$ to the estimate of $(\ref{2.37})$, the additional term's contribution can be bounded by
\begin{equation}\label{3.13}
\lesssim \int_{0}^{T_{0}} \frac{1}{\lambda} \| Iu \|_{H^{1}}^{8} dt \lesssim 1.
\end{equation} 
Therefore, it is possible to apply most of the analysis in $(\ref{2.39.1})$, $(\ref{2.40})$, $(\ref{2.41})$ on $\lambda \mathbb{T}$ for $|T_{0}| \leq \lambda$.\medskip

The one new difficulty is that it is no longer possible to make the simple argument
\begin{equation}
\| u \|_{L_{t,x}^{6}} \lesssim \| u \|_{U_{\Delta}^{3}}.
\end{equation}
Indeed, \cite{bourgain1993fourier} showed that
\begin{equation}
\| e^{it \partial_{xx}} P_{N} \phi \|_{L_{t,x}^{6}([0, \frac{1}{\log N}] \times \mathbb{T})} \lesssim \| \phi \|_{L^{2}},
\end{equation}
is the best such estimate that one could hope for. Recently, \cite{skouloudis2026strichartz} proved
\begin{equation}
\| e^{it \partial_{xx}} P_{N} \phi \|_{L_{t,x}^{6}([0, \frac{1}{(\log N)^{C}}] \times \mathbb{T})} \lesssim \| \phi \|_{L^{2}},
\end{equation}
holds for some constant $C$. Here it is enough to use the Strichartz estimate of \cite{bourgain1993fourier} combined with the above biilinear estimates obtained from the Morawetz estimate. 

\begin{theorem}
There exists some constant $c > 0$ such that or any $N$,
\begin{equation}
\| P_{N} e^{it \partial_{xx}} \phi \|_{L_{t,x}^{6}([0, 1] \times \mathbb{T})} \lesssim \exp(c \frac{\log N}{\log \log N}) \| \phi \|_{L^{2}}.
\end{equation}
\end{theorem}
\begin{proof}
See \cite{bourgain1993fourier}.
\end{proof}
Observe that $\exp(c \frac{\log N}{\log \log N}) \lesssim_{\epsilon} N^{\epsilon}$ for any $\epsilon > 0$. It is easy to apply this result to the torus $\lambda \mathbb{T}$. Redoing the scaling in $(\ref{3.2})$ with $\lambda$ replaced by $\lambda^{-1}$,
\begin{equation}
\| P_{N} e^{it \partial_{x}} \frac{1}{\lambda^{1/2}} \phi(\frac{x}{\lambda}) \|_{L_{t,x}^{6}([0, \lambda^{-2}] \times \lambda \mathbb{T})} = \| P_{\lambda N} e^{it \partial_{x}} \frac{1}{\lambda^{1/2}} \phi(x) \|_{L_{t,x}^{6}([0, 1] \times \mathbb{T})} \lesssim \lambda^{\epsilon} N^{\epsilon} \| \phi \|_{L^{2}} \lesssim N^{\epsilon'} \| \phi \|_{L^{2}},
\end{equation}
where $\epsilon' = \frac{1 - s}{s} \epsilon + \epsilon$. Since $\epsilon > 0$ is arbitrary, it is possible to choose any $\epsilon' > 0$ as well.

\begin{claim}\label{claim}
If $J$ is an interval and $\| \partial_{x} I u_{h} \|_{U_{\Delta}^{3}(J \times \mathbb{T})} \lesssim 1$ with $E(Iu(t)) \leq 1$ on $J$, $u_{h} = P_{\geq M_{j - 1}} u$,
\begin{equation}
\| |\partial_{x} Iu_{h}| |u_{h}|^{2} \|_{L_{t,x}^{2}(J \times \mathbb{T})} \lesssim \frac{1}{M_{j - 1}^{1/2}}.
\end{equation}
\end{claim}
\begin{proof}[Proof of claim]
Decompose
\begin{equation}
\aligned
\| |\partial_{x} Iu_{h}| |u_{h}|^{2} \|_{L_{t,x}^{2}} \lesssim \| |\partial_{x} I P_{\leq 8N} u_{h}| |P_{\leq N} u_{h}|^{2} \|_{L_{t,x}^{2}} + \| |\partial_{x} I P_{> 8N} u_{h}| |P_{\leq N} u_{h}|^{2} \|_{L_{t,x}^{2}} \\
+ \sum_{N_{1} \leq N_{2}, N_{3} \lesssim N_{2}, N_{2} \geq N} \| \partial_{x} I P_{N_{3}} u_{h} \|_{L_{t,x}^{6}} \| P_{N_{2}} u_{h} \|_{L_{t,x}^{6}} \| P_{N_{1}} u_{h} \|_{L_{t,x}^{6}} \\
+ \sum_{N_{1} \leq N_{2}, N_{2} \geq N} \| |P_{N_{1}} u_{h}| |P_{N_{2}} u_{h}| |\partial_{x} I P_{\geq 8N_{2}} u_{h}| \|_{L_{t,x}^{2}}.
\endaligned
\end{equation}

From \cite{bourgain1993fourier} and Bernstein's inequality, along with the fact that $M_{j} \geq N^{\delta/2}$ for any $j$,
\begin{equation}\label{badterm1}
\aligned
\| |\partial_{x} I P_{\leq 8N} u_{h}| |P_{\leq N} u_{h}|^{2} \|_{L_{t,x}^{2}} \lesssim \| \partial_{x} I P_{\leq 8N} u_{h} \|_{L_{t,x}^{6}} \| P_{\leq N} u_{h} \|_{L_{t,x}^{6}} \| P_{\leq N} u_{h} \|_{L_{t,x}^{6}} \lesssim N^{3 \epsilon} \| \partial_{x} Iu_{h} \|_{U_{\Delta}^{3}}^{3} \frac{1}{M_{j - 1}^{2}} \lesssim \frac{1}{M_{j - 1}}.
\endaligned
\end{equation}
Using bilinear estimates and Bernstein's inequality,
\begin{equation}\label{badterm2}
\| |\partial_{x} I P_{> 8N} u_{h}| |P_{\leq N} u_{h}|^{2} \|_{L_{t,x}^{2}} \lesssim \| (\partial_{x} I P_{> 8N} u_{h})(P_{\leq N} u_{h}) \|_{L_{t,x}^{2}} \| P_{\leq N} u_{h} \|_{L_{t,x}^{\infty}} \lesssim \frac{1}{M_{j - 1}^{1/2}}.
\end{equation}
\begin{equation}
\aligned
\sum_{N_{1} \leq N_{2}, N_{3} \lesssim N_{2}, N_{2} \geq N} \| \partial_{x} I P_{N_{3}} u_{h} \|_{L_{t,x}^{6}} \| P_{N_{2}} u_{h} \|_{L_{t,x}^{6}} \| P_{N_{1}} u_{h} \|_{L_{t,x}^{6}} \\ \lesssim \sum_{N_{1} \leq N_{2}, N_{2} \geq N, N_{3} \lesssim N_{2}} N_{3}^{\epsilon} \frac{N_{2}^{\epsilon}}{m(N_{2}) N_{2}} \frac{N_{1}^{\epsilon}}{m(N_{1}) N_{1}} \lesssim \frac{1}{N^{1 - 2 \epsilon}} \frac{1}{M_{j - 1}^{1 - \epsilon}}.
\endaligned
\end{equation}

\begin{equation}
\aligned
\sum_{N_{1} \leq N_{2}, N_{2} \geq N} \| |P_{N_{1}} u_{h}| |P_{N_{2}} u_{h}| |\partial_{x} I P_{\geq 8N_{2}} u_{h}| \|_{L_{t,x}^{2}}  \\ \lesssim \sum_{N_{1} \leq N_{2}, N_{2} \geq N} \| |P_{N_{2}} u_{h}| |\partial_{x} I P_{\geq 8 N_{2}} u_{h}| \|_{L_{t,x}^{2}} \| P_{N_{1}} u_{h} \|_{L_{t,x}^{\infty}} \lesssim \sum_{N_{2} \geq N, N_{1} \leq N_{2}} \frac{1}{m(N_{2}) N_{2}} \frac{N_{1}^{1/2}}{m(N_{1}) N_{1}} \lesssim \frac{1}{N}.
\endaligned
\end{equation}

A similar computation shows that if $\| v \|_{U_{\Delta}^{3}(J \times \mathbb{T})} \lesssim 1$,
\begin{equation}\label{vestimate}
\| |v| |u_{h}|^{2} \|_{L_{t,x}^{2}} \lesssim \frac{1}{M_{j - 1}^{1/2}}.
\end{equation}
\end{proof}
Plugging Claim $\ref{claim}$ into the computations closes the inductive step for the bilinear estimate.\medskip

The computations obtaining $(\ref{3.6.1})$ are similar, and simply utilize Bernstein's inequality, which gives 
\begin{equation}
\| P_{2^{k} N^{\delta}} u \|_{L_{x}^{2}} \lesssim \frac{1}{m(2^{k} N^{\delta}) 2^{k} N^{\delta}} E(Iu(t))^{1/2} \lesssim \frac{1}{m(2^{k} N^{\delta}) 2^{k} N^{\delta}}.
\end{equation}

Now turn to $(\ref{3.7})$. Replacing $I P_{\geq M} u$ with $P_{2^{k} M} e^{it \partial_{xx}} f$ for $\| f \|_{L^{2}} = 1$ gives
\begin{equation}\label{3.12.1}
\aligned
4 \int_{0}^{T_{0}} \int \chi(x) |\partial_{x}(\overline{e^{it \partial_{x}} P_{2^{k} M} f}(t, x) P_{\leq M/8} u(t, x))|^{2} dx dt \\ \lesssim \| P_{2^{k} M} f \|_{H^{1}} \| P_{2^{k} M} f \|_{L^{2}} \| u \|_{L_{t}^{\infty} L_{x}^{2}}^{2} + \| Iu \|_{L_{t}^{\infty} H^{1}} \| u \|_{L_{t}^{\infty} L^{2}} \| P_{2^{k} M} f \|_{L^{2}}^{2} \\
- 2 \int_{0}^{T_{0}} \int Im[\overline{P_{\leq M/8} I u(t, y)} P_{\leq M/8} I \mathcal N] \chi(y) \frac{(x - y)}{|x - y|} \chi(x) Im[\overline{P_{2^{k} M} e^{it \partial_{xx}} f} \partial_{x} P_{2^{k} M} e^{it \partial_{xx}} f](t, x) dx dy dt \\
+ \int_{0}^{T_{0}} \int |e^{it \partial_{xx}} P_{2^{k} M} f(t, y)|^{2} \chi(y) \frac{(x - y)}{|x - y|} \chi(x) Re[\overline{P_{\leq M/8} u} \partial_{x} P_{\leq M/8} \mathcal N](t, x) dx dy \\
- \int_{0}^{T_{0}} \int |e^{it \partial_{xx}} P_{2^{k} M} f(t, y)|^{2}\chi(y) \frac{(x - y)}{|x - y|} \chi(x) Re[\overline{P_{\leq M/8} \mathcal N} \partial_{x} P_{\leq M/8} u](t, x) dx dy \\
+ \int_{0}^{T_{0}} \frac{1}{\lambda} O(\| P_{2^{k} M} f \|_{H^{1}}^{2} \| u \|_{L^{2}}^{2}) dt + \int_{0}^{T_{0}} \frac{1}{\lambda} O(\| P_{2^{k} M} f \|_{L^{2}}^{2} \| Iu \|_{H^{1}}^{2}) + \int_{0}^{T_{0}} \frac{1}{\lambda^{3}} O(\| u \|_{L^{2}}^{2} \| f \|_{L^{2}}^{2}) dt.
\endaligned
\end{equation}

Applying the analysis in $(\ref{2.39.1})$, $(\ref{2.40})$, $(\ref{2.41})$,
\begin{equation}\label{3.12.2}
\aligned
- 2 \int_{0}^{T_{0}} \int Im[\overline{P_{\leq M/8} I u(t, y)} P_{\leq M/8} I \mathcal N] \chi(y) \frac{(x - y)}{|x - y|} \chi(x) Im[\overline{P_{2^{k} M} e^{it \partial_{xx}} f} \partial_{x} P_{2^{k} M} e^{it \partial_{xx}} f](t, x) dx dy dt \\
+ \int_{0}^{T_{0}} \int |e^{it \partial_{xx}} P_{2^{k} M} f(t, y)|^{2} \chi(y) \frac{(x - y)}{|x - y|} \chi(x) Re[\overline{P_{\leq M/8} u} \partial_{x} P_{\leq M/8} \mathcal N](t, x) dx dy \\
- \int_{0}^{T_{0}} \int |e^{it \partial_{xx}} P_{2^{k} M} f(t, y)|^{2}\chi(y) \frac{(x - y)}{|x - y|} \chi(x) Re[\overline{P_{\leq M/8} \mathcal N} \partial_{x} P_{\leq M/8} u](t, x) dx dy \\
\lesssim (M^{-1} N^{\delta/10} + 1) 2^{k} M \| P_{2^{k} M} f \|_{L^{2}}^{2}.
\endaligned
\end{equation}
Also, if $|T_{0}| \leq \lambda$,
\begin{equation}\label{3.12.3}
\aligned
\int_{0}^{T_{0}} \frac{1}{\lambda} O(\| P_{2^{k} M} f \|_{H^{1}}^{2} \| u \|_{L^{2}}^{2}) dt + \int_{0}^{T_{0}} \frac{1}{\lambda} O(\| P_{2^{k} M} f \|_{L^{2}}^{2} \| Iu \|_{H^{1}}^{2}) \\ + \int_{0}^{T_{0}} \frac{1}{\lambda^{3}} O(\| u \|_{L^{2}}^{2} \| f \|_{L^{2}}^{2}) dt
\lesssim 2^{2k} M^{2} \| P_{2^{k} M} f \|_{L^{2}}^{2}.
\endaligned
\end{equation}
Therefore,
\begin{equation}\label{3.12.4}
\aligned
\int_{0}^{T_{0}} \int \chi(x) |\partial_{x}(\overline{e^{it \partial_{x}} P_{2^{k} M} f}(t, x) P_{\leq M/8} u(t, x))|^{2} dx dt \lesssim 2^{2k} M^{2} \| P_{2^{k} M} f \|_{L^{2}}^{2},
\endaligned
\end{equation}
and therefore,
\begin{equation}\label{3.13}
\aligned
\| (P_{2^{k} M} e^{it \partial_{xx}} f)(P_{\leq M/8} u) \|_{L_{t,x}^{2}}^{2} \lesssim \| P_{2^{k} M} f \|_{L^{2}}^{2}.
\endaligned
\end{equation}
Summing over $k \geq 0$,
\begin{equation}\label{3.14}
\| (P_{\geq M} e^{it \partial_{xx}} f)(P_{\leq M/8} u) \|_{L_{t,x}^{2}}^{2} \lesssim 1.
\end{equation}
Finally, if $v \in V_{\Delta}^{3/2} \subset U_{\Delta}^{2}$, $\| v \|_{V_{\Delta}^{3/2}} = 1$, $\hat{v}$ is supported on $|\xi| \geq M$, using $(\ref{3.6.1})$ at the previous level of the induction in the second term in $(\ref{3.15})$,
\begin{equation}\label{3.15}
\aligned
\| v \partial_{x} (P_{\geq M} I(|u|^{4} u)) \|_{L_{t,x}^{1}} \lesssim  M^{-2} N^{\delta/10}.
\endaligned
\end{equation}
Arguing by induction proves $(\ref{3.6})$ and $(\ref{3.7})$. This proves the proposition.

\end{proof}

Now we are ready to prove global well--posedness, Theorem $\ref{t3.1}$. The proof uses a bootstrap argument. Let $J$ be an interval of length $\leq \lambda$ such that $E(Iu(t)) \leq 1$ on $J$.
\begin{equation}\label{3.16}
\aligned
\frac{d}{dt} E(Iu(t)) = \langle Iu_{t}, -\partial_{xx} Iu + |Iu|^{4} (Iu) \rangle = \langle i I(\partial_{xx} u - |u|^{4} u), -\partial_{xx} Iu + |Iu|^{4} (Iu) \rangle \\
= \langle i \partial_{xx} Iu, |Iu|^{4} (Iu) - I(|u|^{4} u) \rangle - \langle i I(|u|^{4} u), |Iu|^{4} (Iu) \rangle.
\endaligned
\end{equation}
The $10$--linear term is easier, so that is the term we will begin with.
\begin{equation}\label{3.17}
\langle -i I(|u|^{4} u), |Iu|^{4} (Iu) \rangle = \langle -i I(|u|^{4} u), |Iu|^{4} (Iu) - I(|u|^{4} u) \rangle.
\end{equation}
Decompose $u = u_{l} + u_{h}$, $u_{l} = P_{\leq N/16} u$. Let $F_{j}$ denote the terms in $|Iu|^{4} (Iu) - I(|u|^{4} u)$ with $j$ $u_{h}'s$ and $5 - j$ $u_{l}$'s.
\begin{equation}\label{3.18}
|Iu_{l}|^{4} (Iu_{l}) - I(|u_{l}|^{4} u_{l}) = F_{0} \equiv 0.
\end{equation}
By Fourier support computations,
\begin{equation}\label{3.19}
\langle -i I(|u_{l}|^{4} u_{l}), F_{1} \rangle \equiv 0.
\end{equation}
Therefore,
\begin{equation}\label{3.20}
\int_{I} \langle -i I(|u|^{4} u), F_{1} \rangle dt \lesssim \| (u_{h})u_{l} \|_{L_{x,t}^{2}}^{2} \| Iu \|_{L_{t}^{\infty} \dot{H}^{1}}^{6} + \| u_{h} \|_{L_{t,x}^{6}}^{6} \| Iu \|_{L_{t}^{\infty} \dot{H}^{1}}^{4} \lesssim \frac{1}{N^{2}}.
\end{equation}
For $F_{2}$, $F_{3}$, $F_{4}$, $F_{5}$, we have at least two $u_{h}$ terms.
\begin{equation}\label{3.21}
\aligned
\int_{I} \langle -i I(|u_{h}|^{4} u_{h}), F_{5} \rangle dt \lesssim \| \partial_{x} I(|u_{h}|^{4} u_{h}) \|_{L_{t}^{3} L_{x}^{1}} \| u_{h} \|_{L_{t,x}^{6}}^{4} \| u_{h} \|_{L_{t}^{\infty} L_{x}^{3}} \\ \lesssim \| \partial_{x} Iu \|_{L_{t}^{\infty} L_{x}^{2}} \| u_{h} \|_{L_{t}^{6} \dot{H}^{1/3, 6}}^{2} \| u_{h} \|_{L_{t}^{\infty} H_{x}^{1/3}}^{2} \| u_{h} \|_{L_{t,x}^{6}}^{4} \| u_{h} \|_{L_{t}^{\infty} L_{x}^{3}} \lesssim \frac{1}{N^{2}}.
\endaligned
\end{equation}
Meanwhile,
\begin{equation}\label{3.22}
\aligned
\int_{I} \langle -i I(|u_{l}|^{4} u_{l}), F_{2} \rangle dt \lesssim \| (u_{h})(u_{l}) \|_{L_{t,x}^{2}}^{2} \| Iu \|_{L_{t}^{\infty} \dot{H}^{1}}^{6} \lesssim \frac{1}{N^{2}}.
\endaligned
\end{equation}
The intermediate terms may be handled in a similar manner.\medskip

Now estimate the $6$--linear term. We consider three separate cases.\medskip

\noindent \textbf{Case 1, $F_{0}$:}
\begin{equation}\label{3.23}
\langle -i I \partial_{xx} u, I(|u_{l}|^{4} u_{l}) - |Iu_{l}|^{4} (Iu_{l}) \rangle \equiv 0.
\end{equation}

\noindent \textbf{Case 2, $F_{1}$:}
In this case we use
\begin{equation}\label{3.24}
|m(\xi + \eta) - m(\xi)| \lesssim \frac{|\eta|}{|\xi|} |m(\xi)|.
\end{equation}
Therefore,
\begin{equation}\label{3.25}
\aligned
\int_{J} \langle -i I \partial_{xx} u, F_{1} \rangle dt \lesssim \frac{1}{N} \| (\partial_{x} I P_{> N/4} u)(P_{\leq N/8} u) \|_{L_{x, t}^{2}}^{2} \| \partial_{x} P_{\leq N^{\delta}} u \|_{L^{\infty}} \| Iu \|_{L^{\infty}} \\
+ \frac{1}{N} \| (\partial_{x} I P_{> N/4} u)(P_{\leq N/8} u) \|_{L_{x,t}^{2}} \| \partial_{x} P_{N^{\delta} \leq \cdot \leq N/8} u \|_{L_{t,x}^{6}} \| Iu \|_{L_{t,x}^{\infty}}^{3/2} \\ \times \| P_{> N/4} u \|_{L_{t,x}^{6}}^{1/2} \| (P_{> N/4} u)(P_{\leq N/8} u) \|_{L_{t,x}^{2}}^{1/2} \lesssim \frac{1}{N^{1 - \delta/2 - 6 \epsilon}}.
\endaligned
\end{equation}

\noindent \textbf{Case 3, $F_{2}$, ..., $F_{5}$:} In this case there are at least two $u_{h}$'s in $F_{j}$, $j = 2, 3, 4, 5$.
In this case,
\begin{equation}\label{3.26}
\aligned
\int_{J} \langle -i \partial_{xx} Iu, F_{2} + ... + F_{5} \rangle dt \lesssim \| (\partial_{x} Iu_{h}) u_{\leq N/64} \|_{L_{t,x}^{2}} \| u_{h} (u_{\leq N/64}) \|_{L_{t,x}^{2}} \| \partial_{x} Iu_{\leq N^{\delta}} \|_{L_{t,x}^{\infty}} \| Iu \|_{L_{t,x}^{\infty}} \\
+ \| (\partial_{x} Iu_{h}) |u_{h}| |P_{\geq N^{\delta}} u| \|_{L_{t,x}^{2}}  \| |\partial_{x} Iu_{\geq N^{\delta}}| |u_{\geq N^{\delta}}|^{2} \|_{L_{t,x}^{2}}  \lesssim \frac{1}{N^{1 - 6 \epsilon}}.
\endaligned
\end{equation}

Thus, if $t \in [0, \lambda^{2} T]$, when $s > \frac{1}{2}$, choosing $\delta(s) > 0$ and $\epsilon(s) > 0$ sufficiently small,
\begin{equation}\label{3.27}
E(Iu(t)) - E(Iu(0)) \lesssim \frac{1}{N^{1 - \delta/2}} \lambda T \sim \frac{N^{\frac{1 - s}{s}} T}{N^{1 - \delta/2 - 6 \epsilon}} \ll 1.
\end{equation}
This proves the Theorem $\ref{t3.1}$.

\section{Modified energy}
To prove Theorem $\ref{t1.1}$, we combine the long time Strichartz estimates in Proposition $\ref{p3.2}$ with an improvement over the energy increment in \cite{li2011global}. The idea is roughly as follows. Suppose it were possible to ignore the contribution of the $6$--linear term
\begin{equation}\label{4.1}
\langle i \partial_{xx} Iu, |Iu|^{4} (Iu) - I(|u|^{4} u) \rangle.
\end{equation}
Then plugging $(\ref{3.20})$ and $(\ref{3.21})$ into $(\ref{3.27})$,
\begin{equation}\label{4.2}
E(Iu(t)) - E(Iu(0)) \lesssim \frac{1}{N^{2}} \lambda T \sim \frac{N^{\frac{1 - s}{s}} T}{N^{2-}} \ll 1,
\end{equation}
when $s > \frac{1}{3}$.\medskip

The goal is to define a modified energy that is close to the energy $E(Iu(t))$, but has a smaller contribution from the $6$--linear terms. The key to doing this is to have a modified energy that provides better cancellation on the non--resonant terms. This modified energy is taken from \cite{li2011global}. The long time Strichartz estimates in Proposition $\ref{p3.2}$ substantially simplify the calculations.\medskip

Let
\begin{equation}\label{4.3}
\{ k_{1}^{\ast}, k_{2}^{\ast}, ..., k_{6}^{\ast} \} = \{ k_{1}, ..., k_{6} \},
\end{equation}
where
\begin{equation}\label{4.4}
| k_{1}^{\ast}| \geq |k_{2}^{\ast}| \geq ... \geq |k_{6}^{\ast}|.
\end{equation}

Now let $\Lambda_{n}(M_{n}, f_{1}, ..., f_{n}) = \Lambda_{n}(f, \bar{f}, f, \bar{f}, ..., f, \bar{f})$ be the multilinear operator
\begin{equation}\label{4.5}
\Lambda_{n}(M_{n}, f_{1}, ..., f_{n}) = \int_{\Gamma_{n}} M_{n}(k_{1}, ..., k_{n}) \prod_{j = 1}^{n} \widehat{f_{j}(k_{j}, t)} (dk_{1})_{\lambda} \cdots (dk_{n - 1})_{\lambda},
\end{equation}
where
\begin{equation}\label{4.6}
\Gamma_{n} = \{ (k_{1}, ..., k_{n}) : k_{1} + ... + k_{n} = 0 \}.
\end{equation}
In general, let 
\begin{equation}\label{4.7}
X_{j}(M_{n}) = M_{n}(k_{1}, ..., k_{j - 1}, k_{j} + ... + k_{j + 4}, k_{j + 5}, ..., k_{n + 4}).
\end{equation}

By direct computation,
\begin{equation}\label{4.8}
E_{I}^{1}(u(t)) = \Lambda_{2}(\sigma_{2}) + \Lambda_{6}(\sigma_{6}), \qquad \sigma_{2} = -\frac{1}{2} m(k_{1}) m(k_{2}) k_{1} k_{2}, \qquad \sigma_{6} = \frac{1}{6} m(k_{1}) \cdots m(k_{6}),
\end{equation}
and
\begin{equation}\label{4.9}
\frac{d}{dt} E_{I}^{1}(u(t)) = \Lambda_{6}(M_{6}) + \Lambda_{10}(M_{10}),
\end{equation}
where
\begin{equation}\label{4.10}
M_{6} = \frac{i}{6} \sum_{j = 1}^{6} (-1)^{j + 1} m^{2}(k_{j}) k_{j}^{2} + \sigma_{6} \alpha_{6} = M_{6}^{1} + M_{6}^{2}, \qquad \alpha_{6} = i \sum_{j = 1}^{6} (-1)^{j} k_{j}^{2}, \qquad \text{and} \qquad M_{10} = i \sum_{j = 1}^{6} (-1)^{j} X_{j}(\sigma_{6}).
\end{equation}

The aim is to decompose
\begin{equation}\label{4.11}
M_{6} = \bar{M}_{6} + \tilde{M}_{6},
\end{equation}
and define the modified energy
\begin{equation}\label{4.11}
E_{I}^{2}(u(t)) = E_{I}^{1}(u(t)) + \Lambda_{6}(\tilde{\sigma}_{6}), \qquad \tilde{\sigma}_{6} = -\frac{\tilde{M}_{6}}{\alpha_{6}}.
\end{equation}
Then by direct computation,
\begin{equation}\label{4.12}
\aligned
\frac{d}{dt} E_{I}^{2}(u(t)) = \frac{d}{dt} E_{I}^{1}(u(t)) + \frac{d}{dt} \Lambda_{6}(\tilde{\sigma}_{6}) = \Lambda_{6}(M_{6}) + \Lambda_{10}(M_{10}) + \Lambda_{6}(\tilde{\sigma}_{6} \alpha_{6}) + i \Lambda_{10}(\sum_{j = 1}^{6} (-1)^{j} X_{j}(\tilde{\sigma}_{6})) \\
= \Lambda_{6}(\bar{M}_{6}) + \Lambda_{6}(\tilde{M}_{6}) - \Lambda_{6}(\tilde{M}_{6}) + \Lambda_{10}(M_{10}) + i \Lambda_{10}(\sum_{j = 1}^{6} (-1)^{j} X_{j}(\tilde{\sigma}_{6})) \\ = \Lambda_{6}(\bar{M}_{6}) + \Lambda_{10}(M_{10}) + i \Lambda_{10}(\sum_{j = 1}^{6} (-1)^{j} X_{j}(\tilde{\sigma}_{6})).
\endaligned
\end{equation}

Again by $(\ref{3.21})$ and $(\ref{3.22})$,
\begin{equation}\label{4.13}
\int_{J} \Lambda_{10}(M_{10}) dt \lesssim \frac{1}{N^{2}}.
\end{equation}
If $|\tilde{\sigma}_{6}| \lesssim 1$, or equivalently, $|\tilde{M}_{6}| \lesssim |\alpha_{6}|$, a similar calculation would show
\begin{equation}\label{4.14}
\int_{J} i \Lambda_{10}(\sum_{j = 1}^{6} (-1)^{j} X_{j}(\tilde{\sigma}_{6})) dt \lesssim \frac{1}{N^{2}}.
\end{equation}

To that end, define
\begin{equation}\label{4.15}
\aligned
\Upsilon &= \{ (k_{1}, ..., k_{6}) \in \Gamma_{6} : |k_{1}^{\ast}| \sim |k_{2}^{\ast}| \gtrsim N \}, \\
\Omega_{1} &= \{ (k_{1}, ..., k_{6}) \in \Upsilon : |k_{3}^{\ast}| \gg |k_{4}^{\ast}| \}, \\
\Omega_{2} &= \{ k_{1}, ..., k_{6} \} \in \Upsilon : |k_{1}|, |k_{3}| \gtrsim N, \qquad \text{or} \qquad |k_{2}|, k_{4}| \gtrsim N, \qquad |k_{3}^{\ast}| \sim |k_{4}^{\ast}| \ll N \}, \\
\Omega_{3} &= \{ k_{1}, ..., k_{6}) \in \Upsilon : |k_{3}^{\ast}| \sim |k_{4}^{\ast}| \ll N, \qquad |k_{1}| \sim |k_{2}| \gtrsim N, \qquad |k_{1}^{2} - k_{2}^{2}| \gg |k_{3}^{2} - k_{4}^{2} + k_{5}^{2} - k_{6}^{2}| \},
\endaligned
\end{equation}
and let $\Omega = \Omega_{1} \cup \Omega_{2} \cup \Omega_{3}$. Next, decompose $M_{6} = \bar{M}_{6} + \tilde{M}_{6}$,
\begin{equation}\label{4.16}
\bar{M}_{6} = (\chi_{\Gamma_{6}} - \chi_{\Omega}) M_{6}^{1} + (\chi_{\Gamma_{6}} - \chi_{\Upsilon}) M_{6}^{2},
\end{equation}
\begin{equation}\label{4.17}
\tilde{M}_{6} = \chi_{\Omega} M_{6}^{1} + \chi_{\Upsilon} M_{6}^{2}.
\end{equation}

\begin{lemma}\label{l4.1}
\begin{equation}\label{4.18}
|\tilde{M}_{6}| \lesssim |\alpha_{6}|.
\end{equation}
\end{lemma}
\begin{proof}
Since $|\frac{M_{2}^{6}}{\alpha_{6}}| = |\sigma_{6}| \lesssim 1$, it suffices to show that
\begin{equation}\label{4.19}
|\chi_{\Omega_{j}} M_{6}^{1}| \lesssim |\alpha_{6}|, \qquad \text{for} \qquad j = 1, 2, 3.
\end{equation}
Abbreviate $\overrightarrow{k} = (k_{1}, k_{2}, k_{3}, k_{4}, k_{5}, k_{6})$. First suppose $\overrightarrow{k} \in \Omega_{1}$. In this case $|k_{3}^{\ast}| \gg |k_{4}^{\ast}|$. Consider two cases separately.\medskip

\noindent \textbf{Case 1:} $\{ k_{1}^{\ast}, k_{2}^{\ast}, k_{3}^{\ast} \} = \{ k_{1}, k_{2}, k_{3} \}$ or $\{ k_{1}, k_{2}, k_{4} \}$. Suppose without loss of generality that $\{ k_{1}^{\ast}, k_{2}^{\ast}, k_{3}^{\ast} \} = \{ k_{1}, k_{2}, k_{3} \}$. In this case, by direct computation,
\begin{equation}\label{4.20}
k_{1}^{2} - k_{2}^{2} + k_{3}^{2} - k_{4}^{2} + k_{5}^{2} - k_{6}^{2} = -2(k_{2} + k_{3})(k_{2} + k_{1}) + O((N_{4}^{\ast})^{2}) \sim N_{1} N_{3}.
\end{equation}
Now then,
\begin{equation}\label{4.21}
|m(k_{4})^{2} k_{4}^{2} - m(k_{5})^{2} k_{5}^{2} + m(k_{6})^{2} k_{6}^{2}| \lesssim (N_{4}^{\ast})^{2}.
\end{equation}
Since $N_{1} \geq N_{3}$,
\begin{equation}\label{4.22}
|m(k_{3})^{2} k_{3}^{2}| \lesssim N_{3}^{2} \lesssim N_{3} N_{1}.
\end{equation}
By the difference of squares,
\begin{equation}\label{4.23}
|m(k_{1})^{2} k_{1}^{2} - m(k_{2})^{2} k_{2}^{2}| = |m(k_{1}) k_{1} - m(k_{2}) k_{2}| |m(k_{1}) k_{1} + m(k_{2}) k_{2}|.
\end{equation}
Since $N_{1} \sim N_{2}$ or $N_{1} \sim N_{3}$, which forces $N_{2} \lesssim N_{1}$,
\begin{equation}\label{4.24}
|m(k_{1}) k_{1} - m(k_{2}) k_{2}| \lesssim N_{1}.
\end{equation}
On the other hand, by the triangle inequality,
\begin{equation}\label{4.25}
|m(k_{1}) k_{1} + m(k_{2}) k_{2}| = |m(k_{1})(k_{1} + k_{2}) + (-m(k_{1}) + m(k_{2}) k_{2}| \lesssim N_{3} + N_{2} |m(k_{2}) - m(k_{1})|.
\end{equation}
Since $m$ is an even function,
\begin{equation}\label{4.26}
|m(k_{2}) - m(k_{1})| = |\int_{k_{1}}^{-k_{2}} m'(t) dt| \lesssim |k_{1} + k_{2}| \frac{1}{N_{1}},
\end{equation}
which proves $|m(k_{1}) k_{1} + m(k_{2}) k_{2}| \lesssim N_{3}$. Therefore, $|m(k_{1})^{2} k_{1}^{2} - m(k_{2})^{2} k_{2}^{2}| \lesssim N_{1} N_{3}$, which implies that $|M_{6}^{1} \chi_{\Omega_{1}}| \lesssim |\alpha_{6}|$ in this case.\medskip

\noindent \textbf{Case 2:} Suppose $\{ k_{1}^{\ast}, k_{2}^{\ast}, k_{3}^{\ast} \} = \{ k_{1}, k_{3}, k_{5} \}$ or $= \{ k_{2}, k_{4}, k_{6} \}$. In this case, $|\alpha_{6}| \gtrsim (N_{1}^{\ast})^{2}$, so clearly $|\chi_{\Omega_{1}} M_{6}^{1}| \lesssim |\alpha_{6}|$.\medskip

Now suppose $\overrightarrow{k} \in \Omega_{2}$. Suppose without loss of generality that $|k_{1}|$, $|k_{3}| \gtrsim N$. In this case, $|\alpha_{6}| \gtrsim N_{1}^{2}$, so in this case as well, $|\chi_{\Omega_{2}} M_{6}^{1}| \lesssim |\alpha_{6}|$.\medskip

Finally suppose $\overrightarrow{k} \in \Omega_{3}$. Since $|k_{1}^{2} - k_{2}^{2}| \gg |k_{3}^{2} - k_{4}^{2} + k_{5}^{2} - k_{6}^{2}|$, $|k_{1}^{2} - k_{2}^{2}| \sim \alpha_{6}$. Since $|k_{3}^{\ast}| \sim |k_{4}^{\ast}| \ll N$,
\begin{equation}\label{4.27}
|m(k_{3})^{2} k_{3}^{2} - m(k_{4})^{2} k_{4}^{2} + m(k_{5})^{2} k_{5}^{2} - m(k_{6})^{2} k_{6}^{2}| = |k_{3}^{2} - k_{4}^{2} + k_{5}^{2} - k_{6}^{2}| \ll |k_{1}^{2} - k_{2}^{2}| \sim \alpha_{6}.
\end{equation}
Finally, since $|k_{1} + k_{2}| \lesssim |k_{3}^{\ast}| \ll |k_{1}| \sim |k_{2}|$,
\begin{equation}\label{4.28}
\aligned
|m(k_{1})^{2} k_{1}^{2} - m(k_{2})^{2} k_{2}^{2}| = |m(k_{1}) k_{1} - m(k_{2}) k_{2}| |m(k_{1}) k_{1} + m(k_{2}) k_{2}| \lesssim |k_{1} - k_{2}| |m(k_{1}) k_{1} + m(k_{2}) k_{2}|.
\endaligned
\end{equation}
By direct computation,
\begin{equation}\label{4.29}
m(k_{1}) k_{1} + m(k_{2}) k_{2} = m(k_{1})(k_{1} + k_{2}) + (-m(k_{1}) + m(k_{2})) k_{2} = O(|k_{1} + k_{2}|) + O(\int_{k_{1}}^{-k_{2}} |m'(t)| dt) |k_{2}| = O(|k_{1} + k_{2}|).
\end{equation}
Therefore,
\begin{equation}\label{4.30}
|m(k_{1})^{2} k_{1}^{2} - m(k_{2})^{2} k_{2}^{2}| \lesssim |k_{1}^{2} - k_{2}^{2}|.
\end{equation}
This proves the lemma, and also implies $(\ref{4.14})$.
\end{proof}

The bound in $(\ref{4.18})$ also implies that the modified energy $E^{2}$ is a good approximation of $E(Iu(t))$.
\begin{lemma}\label{l4.2}
For any $s > 1/3$,
\begin{equation}\label{4.31}
|\Lambda_{6}(\tilde{\sigma}_{6})(t)| \lesssim N^{0-} \| Iu(t) \|_{\dot{H}^{1}}^{6},
\end{equation}
where $0-$ denotes a constant $\epsilon(s) \nearrow 0$ as $s \searrow \frac{1}{3}$, $\epsilon(s) < 0$ for all $s > \frac{1}{3}$.
\end{lemma}
\begin{proof}
By the Lemma $\ref{l4.1}$, to prove $(\ref{4.31})$ it suffices to show that
\begin{equation}\label{4.32}
\int_{\Gamma_{6}} \frac{\hat{f}_{1}(k_{1}) \hat{\bar{f}}_{2}(k_{2}) \hat{f}_{3}(k_{3}) \hat{\bar{f}}_{4}(k_{4}) \hat{f}_{5}(k_{5}) \hat{\bar{f}}_{6}(k_{6})}{\langle k_{1} \rangle m(k_{1}) \langle k_{2} \rangle m(k_{2}) \langle k_{3} \rangle m(k_{3}) \langle k_{4} \rangle m(k_{4}) \langle k_{5} \rangle m(k_{5}) \langle k_{6} \rangle m(k_{6})} \lesssim N^{0-} \| f_{1} \|_{L^{2}} \| f_{2} \|_{L^{2}} \| f_{3} \|_{L^{2}} \| f_{4} \|_{L^{2}} \| f_{5} \|_{L^{2}} \| f_{6} \|_{L^{2}}.
\end{equation}
Let $f_{l} = P_{\leq N/10} f$ and $f_{h} = P_{> N/10} f$. Since $|N_{1}^{\ast}| \sim |N_{2}^{\ast}| \gtrsim N$ on the support of $\tilde{\sigma}_{6}$.
\begin{equation}\label{4.33}
\aligned
\int_{\Gamma_{6}} \frac{\hat{f}_{1}(k_{1}) \hat{\bar{f}}_{2}(k_{2}) \hat{f}_{3}(k_{3}) \hat{\bar{f}}_{4}(k_{4}) \hat{f}_{5}(k_{5}) \hat{\bar{f}}_{6}(k_{6})}{\langle k_{1} \rangle m(k_{1}) \langle k_{2} \rangle m(k_{2}) \langle k_{3} \rangle m(k_{3}) \langle k_{4} \rangle m(k_{4}) \langle k_{5} \rangle m(k_{5}) \langle k_{6} \rangle m(k_{6})} \lesssim \| \mathcal F^{-1}(\frac{f_{l}}{\langle k \rangle m(k)}) \|_{L^{\infty}}^{4} \| \mathcal F^{-1}(\frac{f_{h}}{\langle k \rangle m(k)}) \|_{L^{2}}^{2} \\
+ \| \mathcal F^{-1}(\frac{f_{h}}{\langle k \rangle m(k)}) \|_{L^{6}}^{2} \lesssim \frac{1}{N^{2}} \| f \|_{L^{2}}^{6} + \frac{1}{N^{4}} \| f \|_{L^{2}}^{6}.
\endaligned
\end{equation}
The last estimate follows from $s > 1/3$ and 
\begin{equation}\label{4.34}
\| \frac{f_{h}}{\langle k \rangle m(k)} \|_{L^{6}} \lesssim \| \frac{f_{h}}{\langle k \rangle m(k)} \|_{\dot{H}^{1/3}} \lesssim \| \frac{f_{h}}{\langle k \rangle^{2/3} m(k)} \|_{L^{2}} \lesssim \frac{1}{N^{2/3}} \| f \|_{L^{2}}.
\end{equation}
Since $m(k) \sim |k|^{s - 1} N^{1 - s}$, for $s > 1/3$, $|k|^{2/3} |k|^{s - 1} N^{1 - s} = |k|^{2/3} |k|^{s - 1} N^{1 - s} = |k|^{s - 1/3} N^{1 - s}$.
\end{proof}

It only remains to bound the contribution of $\int_{J} |\Lambda_{6}(\bar{M}_{6})| dt$.
\begin{lemma}\label{l4.3}
For any $\delta > 0$,
\begin{equation}\label{4.35}
\int_{J} |\Lambda_{6}(\bar{M}_{6})| dt \lesssim \frac{1}{N^{2 \delta - 2}}.
\end{equation}
\end{lemma}
\begin{proof}
Since $M_{6}^{1} + M_{6}^{2} = 0$ outside the support of $\Upsilon$, it suffices to compute $(\ref{4.35})$ for $M_{6}^{1}$ on $\Upsilon \setminus \Omega$. Consider two cases separately: $|k_{3}^{\ast}| \sim |k_{4}^{\ast}| \ll N$ and $|k_{1}^{2} - k_{2}^{2}| \lesssim |k_{3}^{2} - k_{4}^{2} + k_{5}^{2} - k_{6}^{2}|$ and $|k_{4}^{\ast}| \gtrsim N$.\medskip

\noindent \textbf{Case 1:} Let $\chi_{1}$ denote the characteristic function of $\Upsilon \setminus \Omega$ restricted to this region. Suppose $|k_{1}^{2} - k_{2}^{2}| \lesssim |k_{3}^{2} - k_{4}^{2} + k_{5}^{2} - k_{6}^{2}| \lesssim (N_{3}^{\ast})^{2}$ and $|k_{3}^{\ast}| \sim |k_{4}^{\ast}|$. Following $(\ref{4.31})$, $|m(k_{1})^{2} k_{1}^{2} - m(k_{2})^{2} k_{2}^{2}| \lesssim (N_{3}^{\ast})^{2}$. Therefore, integrating on this region, by Proposition $\ref{p3.2}$,
\begin{equation}\label{4.36}
\aligned
\int \int \chi_{1} M_{6}^{1} \hat{u}(t, k_{1}) \hat{\bar{u}}(t, k_{2}) \hat{u}(t, k_{3}) \hat{\bar{u}}(t, k_{4}) \hat{u}(t, k_{5}) \hat{\bar{u}}(t, k_{6}) \\
\lesssim \sum_{N^{\delta} \leq N_{3}^{\ast} \sim N_{4}^{\ast} \ll N, N_{1}^{\ast} \sim N_{2}^{\ast} \gtrsim N} (N_{3}^{\ast})^{2} \| (P_{N_{1}^{\ast}} u)(P_{\leq N^{\delta}} u) \|_{L_{t,x}^{2}}^{1/2} \| (P_{N_{2}^{\ast}} u)(P_{\leq N^{\delta}} u) \|_{L_{t,x}^{2}}^{1/2} \| P_{N_{3}^{\ast}} u \|_{L_{t,x}^{6}} \| P_{N_{4}^{\ast}} u \|_{L_{t,x}^{6}} \\
\times \| P_{N_{1}} u \|_{L_{t,x}^{6}}^{1/2} \| P_{N_{2}} u \|_{L_{t,x}^{6}}^{1/2} \| P_{\leq N^{\delta}} u \|_{L_{t,x}^{\infty}}^{1/2} \| P_{\leq N^{\delta}} u \|_{L_{t,x}^{\infty}}^{1/2} \\
+ \sum_{N^{\delta} \leq N_{6}^{\ast} \leq N_{5}^{\ast} \leq N_{4}^{\ast} \sim N_{3}^{\ast} \ll N, N_{1}^{\ast} \sim N_{2}^{\ast} \gtrsim N} (N_{3}^{\ast})^{2} \| P_{N_{1}^{\ast}} u \|_{L_{t,x}^{6}} \| P_{N_{2}^{\ast}} u \|_{L_{t,x}^{6}} \| P_{N_{3}^{\ast}} u \|_{L_{t,x}^{6}} \\ \times \| P_{N_{4}^{\ast}} u \|_{L_{t,x}^{6}} \| P_{N_{5}^{\ast}} u \|_{L_{t,x}^{6}} \| P_{N_{6}^{\ast}} u \|_{L_{t,x}^{6}} \\
+ \sum_{N_{3}^{\ast} \sim N_{4}^{\ast} \leq N^{\delta}} (N_{3}^{\ast})^{2} \| (P_{N_{1}} u)(P_{\leq N_{3}^{\ast}} u) \|_{L_{t,x}^{2}} \| (P_{N_{2}} u)(P_{\leq N_{4}^{\ast}} u) \|_{L_{t,x}^{2}} \| P_{N_{3}^{\ast}} u \|_{L_{t,x}^{\infty}} \| P_{N_{4}^{\ast}} u \|_{L_{t,x}^{\infty}} \lesssim N^{2 \delta + 6 \epsilon - 2}.
\endaligned
\end{equation}

\noindent \textbf{Case 2:} In this case $N_{3}^{\ast} \sim N_{4}^{\ast} \gtrsim N$. Let $\chi_{2}$ denote the characteristic function of this region. In this case, the bound is given by Claim $\ref{claim}$,
\begin{equation}\label{4.37}
\aligned
\sum_{N_{1}^{\ast} \sim N_{2}^{\ast}, N_{3}^{\ast} \sim N_{4}^{\ast} \gtrsim N } m(N_{1}^{\ast}) N_{1}^{\ast} m(N_{2}^{\ast}) N_{2}^{\ast} \| (P_{N_{1}^{\ast}} u)(P_{N_{3}^{\ast}} u) (P_{> N^{\delta}} u) \|_{L_{t,x}^{2}} \| (P_{N_{2}}^{\ast} u)(P_{N_{4}^{\ast}} u)(P_{> N^{\delta}} u) \|_{L_{t,x}^{2}} \\
+ \sum_{N_{1}^{\ast} \sim N_{2}^{\ast}, N_{3}^{\ast} \sim N_{4}^{\ast} \gtrsim N } m(N_{1}^{\ast}) N_{1}^{\ast} m(N_{2}^{\ast}) N_{2}^{\ast} \| (P_{N_{1}^{\ast}} u)(P_{\leq N^{\delta}} u) \|_{L_{t,x}^{2}} \| P_{\leq N^{\delta}} u \|_{L_{t,x}^{\infty}} \| (P_{N_{2}}^{\ast} u)(P_{N_{3}^{\ast}} u)(P_{N_{4}^{\ast}} u) \|_{L_{t,x}^{2}} \\
\lesssim \frac{1}{N^{2 - 2 \delta - 6 \epsilon}}.
\endaligned
\end{equation}
\begin{remark}
Here we do not have any terms like $(\ref{badterm1})$ or $(\ref{badterm2})$.
\end{remark}
\end{proof}

\begin{proof}[Proof of Theorem $\ref{t1.1}$]
Let $u_{0} \in H^{s}(\mathbb{T})$ for some $s > \frac{1}{3}$ and rescale so that $E(Iu(0)) = \frac{1}{2}$. We wish to show that $E(Iu(t)) \leq 1$ on $[0, \lambda^{2} T]$ for some $T < \infty$ fixed. Let $J \subset [0, \lambda^{2} T]$ be a subinterval for which $E(Iu(t)) \leq 1$ on $J$. Partition $J$ into at most $\lambda T$ subintervals $J_{i}$, $|J_{i}| \leq \lambda$ for each $i$. By Lemma $\ref{l4.2}$, $E(Iu(0)) = \frac{1}{2}$ implies $E_{2}(u(0)) \leq \frac{1}{2} + \frac{1}{100}$ for $N(s)$ sufficiently large. Furthermore, by Lemmas $\ref{l4.1}$ and $\ref{l4.3}$, for any $\delta(s) > 0$ fixed,
\begin{equation}\label{4.38}
\int_{J_{i}} |\frac{d}{dt} E_{2}(u(t))| dt \lesssim \frac{1}{N^{2 \delta + 6 \epsilon - 2}}.
\end{equation}
Therefore, for any $t \in J$,
\begin{equation}\label{4.39}
E_{2}(u(t)) \leq \frac{1}{2} + \frac{1}{100} + \frac{\lambda T}{N^{2 \delta - 2}} \leq \frac{1}{2} + \frac{1}{50},
\end{equation}
for $\delta(s) > 0$ sufficiently small, $N(s, T) < \infty$ sufficiently large. Applying Lemma $\ref{l4.2}$ again, $E(Iu(t)) \leq \frac{1}{2} + \frac{3}{100} \leq \frac{3}{4}$. Thus, by standard bootstrap arguments, $E(Iu(t)) \leq 1$ on $[0, \lambda^{2} T]$, which proves Theorem $\ref{t1.1}$.
\end{proof}

\section*{Acknowledgements}
The author gratefully acknowledges many helpful conversations with Ryan McConnell at Johns Hopkins.

\bibliography{biblio}
\bibliographystyle{alpha}
\end{document}